\documentclass[12pt]{amsart}

\usepackage{amssymb,amsfonts,amsmath}
\usepackage[all,arc]{xy}
\usepackage{enumerate}
\usepackage{mathrsfs}
\usepackage{multicol}
\usepackage{xcolor}
\usepackage{quiver}
\usepackage{fullpage}

\usepackage{comment}

\usepackage{hyperref}

\newtheorem{thm}{Theorem}[section]
\newtheorem{cor}[thm]{Corollary}
\newtheorem{prop}[thm]{Proposition}
\newtheorem{lem}[thm]{Lemma}
\newtheorem{conj}[thm]{Conjecture}

\theoremstyle{definition}
\newtheorem{defn}[thm]{Definition}
\newtheorem{exmp}[thm]{Example}

\theoremstyle{remark}
\newtheorem{rem}[thm]{Remark}

\newtheorem{constr}[thm]{Construction}
\newtheorem{claim}[thm]{Claim}

\DeclareMathOperator{\Pic}{Pic}

\DeclareMathOperator{\Mor}{Mor}

\DeclareMathOperator{\NE}{\overline{NE}}

\newcommand{\Nef}{\operatorname{Nef}}

\makeatletter
\let\c@equation\c@thm
\makeatother
\numberwithin{equation}{section}

\title{Rational Curves on Coindex 3 Fano Varieties}%

\author{Eric Jovinelly and Fumiya Okamura}

\begin{document}

\maketitle

\begin{abstract}
    We describe the moduli space of rational curves on smooth Fano varieties of coindex 3.  For varieties of dimension 5 or greater, we prove the moduli space has a single irreducible component parameterizing rational curves of each effective numerical class.  For varieties of dimension 4, we describe families of rational curves in terms of Fujita's $a$-invariant.  
    Our results verify Lehmann and Tanimoto's Geometric Manin's Conjecture for all smooth coindex 3 Fano varieties over the complex numbers.

\end{abstract}

\section{Introduction}\label{section: intro}

%A smooth variety $X$ is called \textit{Fano} if its anticanonical divisor $-K_{X}$ is ample.  %
%Rational curves $f : \mathbb{P}^1 \rightarrow X$ are both abundant \cite{Mori1979} and important to the geometry of a Fano variety.  %
Rational curves on Fano varieties are both abundant \cite{Mori1979} and important to their geometry.  
For instance, studies on rational curves of low anticanonical degree have shown smooth Fano varieties form bounded families in each dimension \cite{KoMiMo1992rationallyconnected} and classified those of low dimension \cite{mori1981classification, mukaiClassifcation, erratumMori}.  
Contemporary research has focused on rational curves of larger degree \cite{Thomsen1998, KimPandharipande2001, Testa2009, beheshti2021rational, Fanoindex1rank1, beheshti2020moduli, lastDelPezzoThreefold, BurkeJovinelly2022, Okamura2024delPezzo}.
This paper describes the moduli of rational curves on higher dimensional analogues of Fano threefolds called smooth coindex 3 Fano varieties.  
These are varieties where $-K_X = (\dim X - 2)H$ for some indivisible Cartier divisor $H$. %
Throughout the paper, we work over the complex number field $\mathbb{C}$.

For each numerical equivalence class $\alpha$ of curves in $X$, we let %
$\Mor(\mathbb{P}^1, X, \alpha)$ denote the quasi-projective morphism scheme parameterizing maps $f : \mathbb{P}^1 \rightarrow X$ such that $f_*[\mathbb{P}^1] = \alpha$.  A map $f$ is called a \textit{free curve} if $f^*\mathcal{T}_X$ is globally generated.  
In the following theorems, $\Nef_1(X)$ (resp.\ $\NE(X)$) denotes the cone of nef (resp.\ effective) curve classes on $X$.  Our primary results classify components of $\Mor(\mathbb{P}^1, X, \alpha)$. %

\begin{thm}\label{thm: main result 1}
    Let $X$ be a smooth Fano variety of coindex 3 and dimension $\geq 5$.  For each $\alpha \in \NE(X)_{\mathbb{Z}}$, $\Mor(\mathbb{P}^1, X, \alpha)$ is irreducible and nonempty.  $\Mor(\mathbb{P}^1, X, \alpha)$ generically parameterizes embedded free curves if $\alpha \in \Nef_1(X) \setminus \partial\NE(X)$.
\end{thm}

When $\dim X = 4$, $\Mor(\mathbb{P}^1, X, \alpha)$ frequently has more than one component.  The structure of these components is similar to the pattern observed by \cite{BurkeJovinelly2022} in dimension 3.  

\begin{thm}\label{thm: main result 2}
    Let $X$ be a smooth Fano variety of coindex $3$ and dimension $4$.  Assume $X$ is not a product variety. %
    Let $-K_X = 2H$ and $\alpha \in \Nef_1(X)_\mathbb{Z}\setminus \partial\NE(X)$. %
    \begin{enumerate}[(1)]
        \item If $H \cdot \alpha \geq 2$, then there exists a unique component of $\Mor(\mathbb{P}^{1}, X, \alpha)$ which generically parameterizing embedded free curves.  Any other component either parameterizes multiple covers of $H$-lines 
        or a non-dominant families of curves. 
        \item If $H \cdot \alpha = 1$ and $X$ is general in moduli, then $\Mor(\mathbb{P}^1,X,\alpha)$ is smooth, irreducible, and generically parameterizes free curves.
    \end{enumerate}
\end{thm}

To explain this pattern, we lean on Lehmann and Tanimoto's Geometric Manin's Conjecture \cite{2019}, motivated by Batyrev's heuristics \cite{batyrev88} for Manin's Conjecture \cite{Manin} over global function fields.  These heuristics assumed $\Mor(\mathbb{P}^1, X, \alpha)$ would have a single component parameterizing free curves when $X$ is an $\mathbb{F}_q$-Fano variety. %
However, this assumption is false in many cases; its failure is related to unexpectedly large point counts on certain varieties \cite{manin_conj_counterexample} which necessitated a revision of Manin's Conjecture \cite{peyre_manin_conjecture}.

Lehmann and Tanimoto sharpened Batyrev's heuristics into a precise conjecture concerning the behavior of rational curves on Fano varieties over arbitrary fields.  While there is no bound on the number of components in $\Mor(\mathbb{P}^1, X, \alpha)$, Lehmann and Tanimoto %
define a subset of these components they call \textit{Manin components} \cite{2019} which they predict will be bounded in number.  Loosely speaking, Manin components avoid certain pathological behavior, such as having a nondominant universal family, or a dominant universal family with disconnected fibers.  
This yields a geometric analogue of the thin set version of Manin's Conjecture \cite{peyre_manin_conjecture}.  An explanation of this analogy and the most general form of Geometric Manin's Conjecture appears in \cite[Section~4]{lehmann2024nonfree}.

\begin{conj}[Geometric Manin's Conjecture, \cite{tanimoto2021introduction}]\label{conj: GMC}
    Let $X$ be a smooth Fano variety with Brauer group $Br(X)$.  %
    There exists $\alpha \in \mathrm{Nef}_{1}(X)_{\mathbb{Z}}$ such that for each $\beta \in \alpha + \mathrm{Nef}_{1}(X)_{\mathbb{Z}}$, $\Mor(\mathbb{P}^{1},X,\beta)$  contains exactly $|\text{Br}(X)|$ Manin components.
    
\end{conj}

Together with \cite{BurkeJovinelly2022}, our results characterize Manin components on any smooth coindex three Fano variety $X$.  When $X$ is not a product of lower-dimensional Fano varieties, Manin components of $\Mor(\mathbb{P}^1,X,\alpha)$ are those which parameterize embedded free curves of class $\alpha \in \Nef_1(X)_\mathbb{Z}\setminus \partial\NE(X)$ with $-K_X \cdot \alpha > 2$.  The following is a corollary of Theorems \ref{thm: main result 1}, \ref{thm: main result 2}, and \cite{BurkeJovinelly2022}.

\begin{cor}\label{thm: GMC}
    Geometric Manin's Conjecture holds for all smooth coindex 3 Fano varieties over a base field of characteristic $0$.
\end{cor}

The framework for Geometric Manin's Conjecture also describes the locus $W \subset X$ which contains the image of all nondominant families of rational curves.  Each irreducible component $Y$ of $W$ is a maximal subvariety on which Fujita's $a$-invariant increases, i.e.\ $a(Y, -K_X|_Y) > a(X, -K_X) = 1$.  The following theorem identifies all such subvarieties in coindex 3 Fano varieties of dimension at least 4.

\begin{thm}\label{classification higher a}
    Let $X$ be a smooth coindex $3$ Fano variety with $\dim X \ge 4$.  Let $Y$ be an irreducible component of the union $W$ of all subvarieties $Z \subset X$ satisfying $a(Z, -K_X|_Z) > 1$.  Then one of the following holds:
    \begin{enumerate}
        \item $\rho(X) = 2$, $\dim X \leq 5$, and $Y$ is the exceptional locus of an elementary divisorial contraction on $X$, or %
        \item $\dim X = 4$, $(Y, -K_X|_Y) \cong (\mathbb{P}^2, \mathcal{O}(2))$, and $Y$ is contracted to a point by an elementary fiber-type contraction on $X$.
    \end{enumerate}
    Furthermore, if $\dim X = 4$ and $X$ is general in moduli, there are no subvarieties of type (2) unless $\rho(X) = 2$ and $g(X) = (\frac{1}{32}K_X^4 + 1) \in \{9,11,14\}$.
\end{thm}

Similar results to Theorems \ref{thm: main result 1} and \ref{thm: main result 2} have been shown for del Pezzo surfaces \cite{Testa2009, beheshti2021rational}, smooth Fano threefolds \cite{Fanoindex1rank1, beheshti2020moduli, BurkeJovinelly2022}, certain Fano hypersurfaces \cite{HRS2004, CS2009, BK2013, BV2017, RY2019}, and del Pezzo varieties of larger dimension \cite{Okamura2024delPezzo}.  A key intermediate result in each of these studies is a so-called movable version of Mori's bend-and-break \cite[Theorem~4]{Mori1979} for the respective varieties.  Such a result allows one to deform (break) free curves of large degree into nodal unions of lower-degree free curves.  We use the Kontsevich space $\overline{M}_{0,0}(X,\alpha)$ of stable maps of class $\alpha$ to record such deformations.

\begin{thm}[Movable Bend-and-Break]\label{MBB}
Let $X$ be a smooth coindex $3$ Fano variety of dimension $n \ge 4$.  Let $-K_X = (n-2)H$.  Suppose $X$ does not contain a contractible divisor.  Let $M \subset \overline{M}_{0,0}(X,\alpha)$ be a component that generically parameterizes free stable maps of $H$-degree at least $2$.  
Then $M$ contains a stable map $f\colon C_{1} \cup C_{2}\rightarrow X$ such that each restriction $f|_{C_{i}}\colon C_{i}\rightarrow X$ is a free morphism.
\end{thm}

Prior methods for proving a Movable Bend-and-Break theorem worked only in low dimensions or when every component of $\overline{M}_{0,0}(X,\alpha)$ generically parameterizes free curves. 
While the latter condition is true for most coindex 3 Fano varieties, $\overline{M}_{0,0}(X,\alpha)$ may have more components than $\Mor(\mathbb{P}^{1}, X, \alpha)$ in general.  We overcome this difficulty using Theorem \ref{classification higher a} and several lemmas about fiber dimensions of families of curves.  Our study results in a complete description of components of $\overline{M}_{0,0}(X,\alpha)$ when $X$ is a coindex 3 Fano variety of dimension 5 or greater.  The following theorem describes all components of $\overline{M}_{0,0}(X,\alpha)$ which do not correspond to a component of $\Mor(\mathbb{P}^{1}, X, \alpha)$.

\begin{thm}\label{thm: component reducible stable maps}
    Let $X$ be a smooth coindex $3$ Fano variety of dimension $n \ge 5$. Let $-K_X = (n-2)H$. Suppose a component $M \subset \overline{M}_{0,0}(X,\alpha)$ generically parameterizes reducible curves.  Let $f : C \to X$ be a general parameterized by $M$.  Then $X$ and $C$ satisfy one of the following:
    \begin{enumerate}
         \item There is a contraction $\pi : X \rightarrow \mathbb{P}^5$ realizing $X$ as the blow-up of $\mathbb{P}^5$ along a line, and a component of $f(C)$ lies in the exceptional divisor of $\pi$; 

         \item $\rho(X) = 1$, $g(X) \leq 3$, $f$ contracts a component $C_0 \subset C$ to a point $p \in X$, and at least $\dim X + 1$ components of $C$ adjacent to $C_0$ are $H$-lines.  Moreover, the family of $H$-lines passing through $p$ has dimension $\dim X - 2$.
    \end{enumerate}
Furthermore, if $g(X) \leq 3$, no such components exist when $X$ is general in moduli.
\end{thm}

In light of Theorem \ref{thm: main result 1}, we obtain the following corollary.

\begin{cor}
    Let $X$ be a smooth coindex 3 Fano variety of dimension at least $6$ and genus greater than $3$.  For each $\alpha \in \NE(X)_{\mathbb{Z}}$, $\overline{M}_{0,0}(X,\alpha)$ is irreducible.
\end{cor}

\subsection{Outline}
The paper is organized as follows.
Section \ref{section: preliminaries} collects the notions which are frequently used in the paper and preliminary results. 
In Section \ref{section: GMC}, we define two geometric invariants $a$, $b$, and state Geometric Manin's Conjecture. 
We also prove Geometric Manin's Conjecture (GMC) for coindex 3 Fano varieties of product type.  
In Section \ref{section: subvar with higher a}, we classify subvarieties of higher $a$-invariants to prove Theorem \ref{classification higher a}.  
In Section \ref{section: low degree curves}, we study the irreducible components of $\overline{M}_{0,0}(X, \alpha)$ for small $\alpha \in \Nef_{1}(X)_{\mathbb{Z}}$.   %
For those $X$ containing a contractible divisor, Lemma \ref{lem: mbb contractible divisor} extends this study to all $\alpha$.
To study spaces of higher degree free curves on the remaining varieties, we prove Movable Bend-and-Break (Theorem \ref{MBB}) in Section \ref{section:mbb}.  
Theorem \ref{thm: component reducible stable maps} is a direct consequence of the detailed study of Kontsevich spaces in this section.  
In Section \ref{section:a covers}, we classify $a$-covers with respect to the Iitaka dimension for the adjoint divisor.
We conclude with a proof of Theorem \ref{thm: main result 1}, Theorem \ref{thm: main result 2}, and Corollary \ref{thm: GMC} in Section \ref{section: proof of Main thms}.   %

\section*{Acknowledgments}
The authors would like to thank Alexander Kuznetsov for answering the questions about linear subspaces on Grassmannians.  The authors would also like to thank Izzet Coskun, Gwyneth Moreland, and Sho Tanimoto for comments on earlier drafts of this paper.  Lastly, the authors would like to thank Brian Lehmann, Eric Riedl, and Sho Tanimoto for an introduction to the topic and for their continued support and encouragement.  The first author was supported by an NSF postdoctoral research fellowship, DMS-2303335.  The second author was partially supported by JST FOREST program Grant number JPMJFR212Z and JSPS Bilateral Joint Research Projects Grant number JPJSBP120219935.

\section{Preliminaries}\label{section: preliminaries}
Let $X$ be a smooth projective variety.
Let $N_{1}(X)$ denote the $\mathbb{R}$-vector space of $\mathbb{R}$-$1$-cycles on $X$ modulo numerical equivalence. 
The dimension of $N_{1}(X)$ is called the \textit{Picard rank} $\rho(X)$.
Let $\NE(X)$ be the closure of the cone generated by effective $1$-cycles on $X$ and $\Nef_{1}(X)$ the closed cone generated by \textit{nef} $1$-cycles, i.e., $1$-cycles which have non-negative intersections with effective divisors on $X$. 

For $\alpha \in \NE(X)$, let $\overline{M}_{0,0}(X, \alpha)$ be the union of components of the Kontsevich space $\overline{M}_{0,0}(X)$ parameterizing stable maps $f\colon C\rightarrow X$ with $f_{*}C = \alpha$. 
For a stable map $f\colon C\rightarrow X$ of genus $0$ which is an immersion around each node, the \textit{normal sheaf} of $f$ is defined by the sequence 
\[0\rightarrow \mathcal{E}xt^{1}_{\mathcal{O}_{C}}(Q, \mathcal{O}_{C}) \rightarrow N_{f}\rightarrow \mathcal{H}om_{\mathcal{O}_{C}}(K, \mathcal{O}_{C}) \rightarrow 0, \]
where $K$ (resp.~$Q$) is the kernel (resp.~cokernel) of the map $f^{*}\Omega_{X}\rightarrow \Omega_{C}$.
By deformation theory (e.g., \cite{Behrend1997normalcone}, \cite{GHSrational}), the Zariski tangent space of $\overline{M}_{0,0}(X, \alpha)$ at $[f]$ is given by $H^{0}(X, N_{f})$ and the obstruction space is given by $H^{1}(X, N_{f})$, 
Hence the dimension of $\overline{M}_{0,0}(X, \alpha)$ at $[f]$ is at least $-K_{X}\cdot \alpha + \dim X - 3$, which is often called the \textit{expected dimension}.

\begin{defn}
    Let $X$ be a smooth projective variety. 
    We say that a non-constant morphism $f\colon \mathbb{P}^{1}\rightarrow X$ is \textit{free} (resp. \textit{very free}) if it satisfies $H^{1}(\mathbb{P}^{1}, f^{*}T_{X}(-1)) = 0$ (resp. $H^{1}(\mathbb{P}^{1}, f^{*}T_{X}(-2)) = 0$).
\end{defn}
In particular, a free morphism corresponds to a smooth point of $\overline{M}_{0,0}(X)$.

\begin{defn}
    Let $X$ be a smooth projective variety and let $C$ be a stable curve of genus $0$.
    \begin{itemize}
        \item A quasi-stable curve $C$ is a \textit{chain of rational curves} of length $d$ if $C$ is a union of $d$ rational curves $C_{1}, \dots, C_{d}$ such that for each $i\in \{1,\dots,d-1\}$, $C_{i+1}$ intersects $C_{1} \cup \dots \cup C_{i}$ at a single point contained in $C_{i}$.

        \item A stable map $f\colon C\rightarrow X$ is a \textit{chain of free curves} on $X$ if $C$ is a chain of rational curves and each restriction $f|_{C_{i}}$ is free.

        \item A quasi-stable curve $C$ is a \textit{comb} if there exists a component $D$, called the \textit{handle}, such that all nodes of $C$ lie in $D$. %
    \end{itemize}
\end{defn}

\begin{lem}[\cite{GHSrational}, Lemma 2.6]
    Let $X$ be a smooth projective variety.
    Let $f\colon C\rightarrow X$ be a genus $0$ stable map which is an immersion near the nodes of $C$. 
    Take a component $D$ of $C$ and let $C_{1}, \dots, C_{d}$ be (the closure of) connected components of $C\setminus D$.
    Then there is an exact sequence
    \[0\rightarrow N_{g}\rightarrow N_{f}|_{D}\rightarrow \bigoplus_{i=1}^{d}k(p_{i})\rightarrow 0, \]
    where $g := f|_{D}$ and $p_{i} \in D$ is the point of intersection with each $C_{i}$.
\end{lem}

\begin{lem}[\cite{Kollar1996}, Lemma 7.5]
    Let $X$ be a smooth projective variety and let $E$ be a vector bundle on $X$.
    Suppose $f\colon C = C_{1}\cup \dots \cup C_{d}\rightarrow X$ is a stable map such that $E|_{C_{i}}$ is globally generated for $i\ge 2$.
    Then $H^{1}(C, E) \cong H^{1}(C_{1}, E|_{C_{1}})$. 
\end{lem}
Applying the lemma to a chain of free curves $f\colon C\rightarrow X$, we obtain the vanishing $H^{1}(C, f^{*}T_{X}) = 0$.
Hence, in particular, $[f] \in \overline{M}_{0,0}(X)$ is a smooth point.

\subsection{Classification of coindex $3$ Fano varieties} 
Fano varieties are often grouped by the divisibility of their canonical divisor relative to their dimension.  Definition \ref{def: index} introduces several commonly used terms to describe this property.

\begin{defn}\label{def: index}
Let $X$ be a smooth Fano variety of dimension $n$.
\begin{itemize}
    \item The \textit{index} $r_X$ of $X$ is the largest integer such that $-K_X = r_X H$ for some ample Cartier divisor $H$.  $H$ is called the \textit{fundamental divisor class}.
    \item The \textit{coindex} of $X$ is $\text{dim }X + 1 - r_X$.
    \item If $X$ is of coindex 3, the \textit{genus} of $X$ is $g = \frac{1}{2}H^n + 1$.
    \item The \textit{fundamental linear system} of $X$ is the complete linear system $|H|$, where $-K_X = r_X H$. If $X$ is of coindex 3, $\text{dim }H^0(X, H) = n + g - 1$.
    
\end{itemize}
\end{defn}

In this paper, we focus on smooth coindex $3$ Fano varieties. 
They are classified by Mukai as follows: 
\begin{thm}[\cite{mukaiClassifcation}]\label{classification Picard rank 1}
Let $X$ be a smooth Fano variety of coindex $3$, dimension $n \geq 4$, and Picard rank $\rho(X) = 1$.  Let $H$ be the fundamental divisor and $g := \frac{1}{2}H^n + 1$ be the genus of $X$.
\begin{enumerate}[(1)]
    \item If $H$ is not very ample, $X$ is isomorphic to one of the following varieties:
    \begin{enumerate}[a.]
        \item $(g = 2)$: a weighted hypersurface of degree $6$ in $\mathbb{P}(1^{n+1}, 3)$;
        \item $(g = 3)$: a weighted complete intersection of type $(4,2)$ in $\mathbb{P}(1^{n+2}, 2)$;
    \end{enumerate}
    \item If $H$ is very ample and $g \leq 5$, $X \subset \mathbb{P}|H|$ is one of the following types of complete intersections:
    \begin{enumerate}[a.]
        \item $(g = 3)$: $X \subset \mathbb{P}^{n+1}$ is a quartic hypersurface. 
        \item $(g = 4)$: $X \subset \mathbb{P}^{n+2}$ is a complete intersection of type $(2,3)$;
        \item $(g = 5)$: $X \subset \mathbb{P}^{n+3}$ is a complete intersection of type $(2,2,2)$;
    \end{enumerate}
    \item If $g \geq 6$, $X$ is a linear section of one of the following varieties:
    \begin{enumerate}[a.]
        \item $(g = 6)$: a quadric section of the cone over the Grassmannian $G(2,5) \subset \mathbb{P}^9$ in the Pl\"{u}cker embedding;
    \item $(g = 7)$: the spinor variety $OG_+(5,10) \subset \mathbb{P}^{15}$ in the embedding induced by the half-spinor representation;
    \item $(g = 8)$: the Grassmannian $G(2,6) \subset \mathbb{P}^{14}$ in the Pl\"{u}cker embedding;
    \item $(g = 9)$: the symplectic Grassmannian $SG(3,6) \subset \mathbb{P}^{13}$ in the Pl\"{u}cker embedding;
    \item $(g = 10)$: a closed orbit $G_2/P \subset \mathbb{P}^{13}$ of the adjoint representation of $G_2$.
    \end{enumerate}
\end{enumerate}
\end{thm}

If the fundamental linear system $|H|$ of a Fano variety $X$ is very ample, the embedding $X \rightarrow \mathbb{P}|H|$ is called the \textit{fundamental model} of $X$.  Each coindex 3 Fano variety listed in Theorem \ref{classification Picard rank 1}(2-3) is embedded via its fundamental linear system.

\begin{thm}[\cite{mukaiClassifcation}]\label{classification Picard rank 2}
Let $X$ be a smooth Fano variety of coindex $3$, dimension at least $4$, and Picard rank $\rho(X) \ge 2$.  
Suppose that $X$ is not a Fano $4$-fold of product type, i.e., $X$ is not the product of $\mathbb{P}^{1}$ and a Fano $3$-fold of even index. 
Then $X$ is isomorphic to a linear section of the fundamental model of one of the following Fano varieties of Picard rank $2$:
\begin{itemize}
    \item $(g = 7)$: a double cover of $\mathbb{P}^2 \times \mathbb{P}^2$ branched over a divisor of bidegree $(2,2)$;
    \item $(g = 9)$: a divisor in $\mathbb{P}^2 \times \mathbb{P}^3$ of bidegree $(1,2)$;
    \item $(g = 11)$: $\mathbb{P}^3 \times \mathbb{P}^3$;
    \item $(g = 11)$: $\mathbb{P}^2 \times Q^3$;
    \item $(g = 12)$: The blow-up of a smooth quadric $Q^4 \subset \mathbb{P}^5$ along a conic not contained in a plane lying in $Q^4$;
    \item $(g = 13)$: the flag variety of $\textit{Sp}(2)$ (or equivalently, of $\textit{SO}(5)$);
    \item $(g = 14)$: The blow-up of $\mathbb{P}^5$ along a line;
    \item $(g = 16)$: the $\mathbb{P}^1$-bundle $\mathbb{P}(\mathcal{O}_Q \oplus \mathcal{O}_Q(1))$ over $Q^3 \subset \mathbb{P}^4$;
    \item $(g = 21)$: the $\mathbb{P}^1$-bundle $\mathbb{P}(\mathcal{O}_\mathbb{P} \oplus \mathcal{O}_\mathbb{P}(2))$ over $\mathbb{P}^3$.

\end{itemize}
\end{thm}

\section{Geometric Manin's Conjecture}\label{section: GMC}
Our study is motivated by Geometric Manin's conjecture (GMC) introduced by Lehmann and Tanimoto \cite{2019}.
GMC predicts the asymptotic formula for the number of components of $\Mor(\mathbb{P}^{1}, X, d)$ after removing ``pathological components.''
There are two geometric invariants which play a central role in GMC: the $a$-invariant and the $b$-invariant.
The $a$-\textit{invariant} of a smooth projective variety $X$ with a nef and big divisor $L$ is defined as 
\[a(X,L) = \inf\{t\in \mathbb{R} \mid K_{X} + tL \mbox{ is a psuedo-effective divisor}\}. \]
By \cite{BDPP2013}, the condition $a(X,L) > 0$ is equivalent that $X$ is uniruled.
In this case, the $b$-\textit{invariant} is defined as 
\[b(X, L) = \dim \langle F(X,L) \rangle, \]
where $\langle \cdot \rangle$ denotes the linear span, and $F(X,L)$ is the convex subcone of the nef cone $\mathrm{Nef}_{1}(X)$ consisting of nef $1$-cycle class vanishing against $K_{X} + a(X,L)L$.
By \cite{balancedlinebundles}, the $a$- and $b$-invariants are birational invariants.
Thus these are defined for singular projective varieties pulling back to a smooth resolution.
\par
Pathological components of $\Mor(\mathbb{P}^{1},X)$ are closely related to generically finite morphism $f\colon Y\rightarrow X$ with $a(Y, -f^{*}K_{X}) \ge a(X, -K_{X})$.
For instance, the following facts are known:
\begin{prop}[\cite{2019}, Proposition 4.2]\label{prop: nondominant curves subvariety a-invar}
    Let $X$ be a smooth Fano variety.
    Suppose that an irreducible component $M$ of $\Mor(\mathbb{P}^{1},X)$ is a non-dominant component, that is, the associated evaluation morphism $\mathrm{ev}\colon \mathcal{U}\rightarrow X$ from the universal family of $M$ is non-dominant.
    Then the closure $Z$ of the image of $\mathrm{ev}$ satisfies $a(Z, -K_{X}|_{Z}) > a(X, -K_{X})$. 
\end{prop}

\begin{prop}[\cite{2019}, Proposition 5.15]\label{prop: disconnected fibers a-cover}
    Let $X$ be a smooth Fano variety. 
    Let $M$ be an irreducible component of $\Mor(\mathbb{P}^{1}, X)$ with the universal family $\mathcal{U}\rightarrow M$ and take a smooth resolution $\phi\colon \tilde{\mathcal{U}}\rightarrow \mathcal{U}$.
    Suppose that the general fiber of the associated evaluation morphism $\mathrm{ev}\colon \mathcal{U}\rightarrow X$ is not irreducible.
    Then the finite part $f\colon Y\rightarrow X$ of the Stein factorization of $\mathrm{ev} \circ \phi$ satisfies $a(Y, -f^{*}K_{X}) = a(X, -K_{X})$.
\end{prop}

\begin{defn}
    A dominant map $f\colon Y\rightarrow X$ is called an \textit{$a$-cover} if $f$ is not birational and $a(Y, -f^{*}K_{X}) = a(X, -K_{X})$.  Conjecturally, each $a$-cover admits a rational map to the finite part of some Stein factorization as in Proposition \ref{prop: disconnected fibers a-cover}.
\end{defn}  

Based on such phenomena, the authors of \cite{2019} introduced accumulating components and Manin components as follows:
\begin{defn}\label{def: accumulating component}
    Let $X$ be a smooth Fano variety.
    \begin{enumerate}[(1)]
        \item We say that a generically finite morphism $f\colon Y\rightarrow X$ from a smooth projective variety $Y$ is \textit{breaking} if either 
        \begin{itemize}
            \item $a(Y, -f^{*}K_{X}) > a(X, -K_{X})$, 
            \item $f$ is an $a$-cover %
            and $\kappa(Y, K_{Y}-f^{*}K_{X}) > 0$, or 
            \item $f$ is an $a$-cover %
            and the induced map $F(Y, -f^{*}K_{X})\rightarrow F(X, -K_{X})$ is not injective.
        \end{itemize}
        \item We say that an irreducible component $M$ of $\Mor(\mathbb{P}^{1}, X)$ is \textit{accumulating} if there exists a breaking morphism $f\colon Y\rightarrow X$ which induces a dominant map $N\dashrightarrow M$ from some irreducible component of $\Mor(\mathbb{P}^{1}, Y)$.
        \item If an irreducible component $M$ is not accumulating, then we say that $M$ is a \textit{Manin component}.
    \end{enumerate}
\end{defn}

Finally one can state Geometric Manin's Conjecture: 
\theoremstyle{plain}
\newtheorem*{MainConj}{\rm\bf Conjecture \ref{conj: GMC}}
\begin{MainConj}
    Let $X$ be a smooth Fano variety with Brauer group $Br(X)$.  %
    There exists $\alpha \in \mathrm{Nef}_{1}(X)_{\mathbb{Z}}$ such that for each $\beta \in \alpha + \mathrm{Nef}_{1}(X)_{\mathbb{Z}}$, $\Mor(\mathbb{P}^{1},X,\beta)$  contains exactly $|\text{Br}(X)|$ Manin components.
\end{MainConj}

For example, Geometric Manin's Conjecture is known for homogeneous spaces (\cite{Thomsen1998}, \cite{KimPandharipande2001}) and toric varieties (\cite{Bourqui2012}, \cite{Bourqui2016}). 
The methods used in these papers are different from ours.
We mainly follow the technique pioneered in \cite{HRS2004}, where they prove the irreducibility for low degree hypersurfaces.
The technique has been developed %
in many papers such as \cite{Testa2005}, \cite{Testa2009}, \cite{CS2009}, \cite{BK2013},\cite{BV2017}, \cite{RY2019}, \cite{2019}, \cite{Fanoindex1rank1}, \cite{lastDelPezzoThreefold}, \cite{beheshti2020moduli}, \cite{BurkeJovinelly2022}, \cite{Okamura2024delPezzo}.

\begin{rem}
    The authors of \cite{sectionsDelPezzo} expect that the number of Manin components of each $\Mor(\mathbb{P}^{1},X,\beta)$ is equal to the size of the Brauer group $\mathrm{Br}(X)$. 
    This expectation is based on the conjecture that $|\mathrm{Br}(X)| = |\mathrm{Ker}(B_{1}(X)_{\mathbb{Z}}\rightarrow N_{1}(X)_{\mathbb{Z}})|$ for smooth rationally connected variety $X$, where $B_{1}(X)_{\mathbb{Z}}$ denotes the set of algebraically equivalent classes of $1$-cycles. 
    This conjecture is known in dimension at most $3$. 
    Moreover, It is known that the Brauer group is trivial for any  smooth Fano threefold. 
    Since general linear sections of smooth coindex $3$ Fano varieties are smooth Fano threefolds, the Lefschetz hyperplane theorem implies the Brauer groups of smooth coindex $3$ Fano varieties are trivial.
\end{rem}

\subsection{Product varieties}
In this subsection, we consider Geometric Manin's Conjecture for product varieties.
Combining the studies on del Pezzo threefolds \cite{2019}, \cite{lastDelPezzoThreefold}, \cite{BurkeJovinelly2022}, we conclude Geometric Manin's Conjecture for arbitrary smooth Fano fourfolds of product type.  Below, we call an $a$-cover $f\colon Y\rightarrow X$ \textit{adjoint rigid} if $\kappa(Y, K_{Y}-f^{*}K_{X}) = 0$.  We say $f$ is \textit{face contracting} if the induced map $F(Y, -f^{*}K_{X})\rightarrow F(X, -K_{X})$ is not injective.

\begin{lem}\label{lem: GMC for product varieties}
    Let $X, Y$ be smooth weak Fano varieties.
    Suppose that Geometric Manin's Conjecture holds for $X$ and $Y$.
    If $X$ and $Y$ do not have adjoint rigid $a$-covers which are not face contracting, then Geometric Manin's Conjecture holds for $X\times Y$. 
\end{lem}

\begin{proof}
    Clearly, $N_{1}(X\times Y) \cong N_{1}(X)\oplus N_{1}(Y)$. 
    Then we see that $\mathrm{Mor}(\mathbb{P}^{1},X\times Y, \alpha+\beta) \cong \mathrm{Mor}(\mathbb{P}^{1},X,\alpha) \times \mathrm{Mor}(\mathbb{P}^{1},Y,\beta)$ for any $\alpha \in N_{1}(X)$, $\beta\in N_{1}(Y)$. 
    Let $\alpha_{1} \in \mathrm{Nef}_{1}(X)$ (resp. $\alpha_{2}\in \mathrm{Nef}_{1}(Y))$ be as in Conjecture \ref{conj: GMC}. 
    Let $M = M_{1}\times M_{2} \subset \mathrm{Mor}(\mathbb{P}^{1},X\times Y, \beta_{1} + \beta_{2})$ be a component, where $\beta_{1}\in \alpha_{1}+\mathrm{Nef}_{1}(X)$ and $\beta_{2}\in \alpha_{2}+\mathrm{Nef}_{1}(Y)$. 
    
    We claim that $M$ is accumulating if and only if either $M_{1}$ or $M_{2}$ is accumulating.
    Suppose $M_{1}$ is accumulating, i.e., there is a breaking morphism $g\colon V\rightarrow X$ and a component $N_{1}\subset \mathrm{Mor}(\mathbb{P}^{1},V)$ such that $f$ induces a dominant map $g_{*}\colon N_{1}\dashrightarrow M_{1}$. 
    Then the base change $g\times \mathrm{id}_{Y}\colon V\times Y\rightarrow X\times Y$ is a breaking morphism, which induces a dominant map $N_{1}\times M_{2}\dashrightarrow M_{1}\times M_{2}$. 
    
    Suppose contrary that $M$ is accumulating. 
    If $M$ is a non-dominant component, then $M_{1}$ or $M_{2}$ is a non-dominant component.
    Hence we may assume that there is an $a$-cover $f\colon W\rightarrow X\times Y$ and a dominant component $N\subset \mathrm{Mor}(\mathbb{P}^{1},W)$ such that the induced map $f_{*}\colon N\dashrightarrow M$ is dominant.
    \begin{claim}\label{claim: factorization of a-covers}
        Either 
        \begin{itemize}
            \item the finite part of the Stein factorization of $\mathrm{pr}_{Y}\circ f$, or
            
            \item the restriction of $f$ to an irreducible component of a general fiber of $\mathrm{pr}_{Y}\circ f$
        \end{itemize}
        is an $a$-cover of $Y$ or $X$.
    \end{claim}
    \begin{proof}
        Let $g\colon V\rightarrow Y$ be the finite part of the Stein factorization of $\mathrm{pr}_{Y}\circ f$.
        Then by a similar argument of \cite[Proposition 5.15]{2019}, one can show that $a(V, -g^{*}K_{Y}) = 1$. 
        Hence we may assume that the fiber $W_{y}$ of $\mathrm{pr}_{Y}\circ f$ over a general point $y\in Y$ is irreducible and smooth. 
        
        We will show that the restriction $f_{y}\colon W_{y}\rightarrow X$ of $f$ is an $a$-cover. 
        Consider the components $\overline{M}\subset \overline{M}_{0,0}(X\times Y, \beta_{1}\times \beta_{2})$, $\overline{M}_{1}\subset \overline{M}_{0,0}(X, \beta_{1})$, $\overline{M}_{2}\subset \overline{M}_{0,0}(Y,\beta_{2})$ and $\overline{N}\subset \overline{M}_{0,0}(W,\gamma)$ corresponding to $M$, $M_{1}$, $M_{2}$ and $N$. 
        There is a divisor $\Delta \subset \overline{M}$ isomorphic to $(\overline{M}'_{1}\times Y)\times_{X\times Y} (X\times \overline{M}'_{2})$. 
        Since the induced map $f_{*}\colon \overline{N}\rightarrow \overline{M}$ is dominant, there is a divisor $\Gamma \subset \overline{N}$ dominating $\Delta$ via $f_{*}$. 
        Hence there is a component $\overline{N}_{1} \subset \overline{M}_{0,0}(W_{y}, \gamma_{1})$ for some $\gamma_{1}\in \mathrm{Nef}_{1}(W_{y})$ such that $f_{y}$ induces a dominant map $\overline{N}_{1}\rightarrow \overline{M}_{1}$.
        Since $\overline{N}_{1}$ is a dominant component, we have $\dim \overline{N}_{1} = -K_{W_{y}}\cdot \gamma_{1} + \dim W_{y} - 3$. 
        Combining with $\dim \overline{N}_{1} = \dim \overline{M}_{1} = -K_{X}\cdot \beta_{1} + \dim X - 3$, we obtain that $(K_{W_{y}}-f_{y}^{*}K_{X})\cdot \gamma_{1} = 0$. 
        Therefore, $a(W_{y}, -f_{y}^{*}K_{X}) = 1$, and $f_{y}$ is an $a$-cover. 
    \end{proof}
    Since we have assumed that $X$ and $Y$ do not have adjoint rigid $a$-covers which are not face contracting, $a$-covers constructed in Claim \ref{claim: factorization of a-covers} are breaking morphisms. Hence $M_{1}$ or $M_{2}$ is accumulating. 
    In other words, $M = M_{1}\times M_{2} \subset \mathrm{Mor}(\mathbb{P}^{1},X\times Y)$ is a Manin component if and only if $M_{1}$ and $M_{2}$ are Manin components. 
    Thus Geometric Manin's Conjecture holds for $X\times Y$.
\end{proof}

\begin{rem}
    Claim \ref{claim: factorization of a-covers} does not necessarily hold for any $a$-cover $f\colon W\rightarrow X\times Y$. 
    For example, take a finite morphism $V\rightarrow X$ which is not an $a$-cover, and consider the base change $f\colon W = V\times Y\rightarrow X\times Y$. 
    Then $f$ is an $a$-cover, however Claim \ref{claim: factorization of a-covers} fails.
    
\end{rem}

\begin{exmp}\label{example: products accumulating}
    Let $V\subset \mathbb{P}^{4}$ be a smooth cubic hypersurface, and set $X = \mathbb{P}^{1}\times V$. 
    Then there exists an accumulating component of $\Mor(\mathbb{P}^{1},X)$ which generically parametrizes embedded free curves of arbitrarily high degree.
    Indeed, for any $d>0$, we have the unique component $M$ parametrizing curves of bidegree $(d,1)$. 
    Since the space $\Mor(\mathbb{P}^{1},V,1)$ is irreducible and accumulating, we see that $M$ is also accumulating.

\end{exmp}

We will use Lemma \ref{lem: GMC for product varieties} and the following to prove Geometric Manin's Conjecture for coindex 3 Fano varieties of product type.

\begin{lem}\label{lem: transitivity of a-covers}
    Let $Y$ and $Z$ be $\mathbb{Q}$-factorial terminal weak Fano varieties. 
    Let $g\colon X\rightarrow Y$ and $h\colon Y\rightarrow Z$ be dominant generically finite morphisms. 
    If $f:=h\circ g$ is an $a$-cover of $Z$, then $g$ (resp. $h$) is either an $a$-cover of $Y$ (resp. $Z$) or a birational morphism.  Moreover, in this case $\kappa(X, K_{X}-g^{*}K_{Y}), \kappa(Y, K_{Y}-h^{*}K_{Z}) \le \kappa(X, K_{X}-f^{*}K_{Z})$. 
\end{lem}
\begin{proof}
    The lemma follows from 
    \[
    K_{X}-f^{*}K_{Z} = K_{X}-g^{*}K_{Y} + g^{*}(K_{Y}-h^{*}K_{Z}) \in \partial\overline{\mathrm{Eff}}^{1}(X)
    \]
    and the fact that $K_{X}-f^{*}K_{Z}$, $K_{X}-g^{*}K_{Y}$ and $K_{Y}-h^{*}K_{Z}$ are the ramification divisors of $f$, $g$ and $h$, in particular, effective. 
\end{proof}

As a particular case of Lemma \ref{lem: GMC for product varieties}, we obtain the following corollary, which can be applied to wider cases.

\begin{cor}\label{cor: product varieties GMC}
    Let $X_{1},\dots, X_{k}$ be smooth Fano varieties of dimension $\le 3$. 
    Then Geometric Manin's Conjecture holds for $X_{1}\times \dots \times X_{k}$. 
\end{cor}

\begin{rem}
    By \cite{Testa2009}, \cite{2019}, \cite{lastDelPezzoThreefold}, \cite{BurkeJovinelly2022}, it is known Geometric Manin's Conjecture holds for any smooth Fano variety of dimension at most $3$.
\end{rem}

\begin{proof}
    By \cite[Theorem 6.2]{Lehmann_2017} and \cite[Section 5]{beheshti2020moduli}, any smooth Fano variety of dimension at most $3$ does not have adjoint rigid $a$-covers. 
    By Lemma \ref{lem: GMC for product varieties}, it suffices to show that $X_{1}\times X_{2}$ does not have adjoint rigid $a$-covers $f\colon W\rightarrow X\times Y$ such that $f$ induces a dominant map $N\dashrightarrow M$ for some components $N\subset \mathrm{Mor}(\mathbb{P}^{1},W)$, $M\subset \mathrm{Mor}(\mathbb{P}^{1},X\times Y)$.
    
    Suppose that the Stein factorization $g\colon V\rightarrow Y$ of $\mathrm{pr}_{Y}\circ f$ is non-trivial. 
    Now there is a diagram
\[\begin{tikzcd}
	W & {X\times V} & {X\times Y} \\
	& V & Y
	\arrow["h", from=1-1, to=1-2]
	\arrow["{\mathrm{id}_{X}\times g}", from=1-2, to=1-3]
	\arrow[from=1-2, to=2-2]
	\arrow[from=1-3, to=2-3]
	\arrow["g", from=2-2, to=2-3]
\end{tikzcd}\]
    such that $f = (\mathrm{id}_{X} \times g)\circ h$. 
    Then by Lemma \ref{lem: transitivity of a-covers}, we have 
    \begin{align*} 
        \kappa(W, K_{W}-f^{*}K_{X\times Y}) &\ge \kappa(X\times V, K_{X\times V}-(\mathrm{id}_{X} \times g)^{*}K_{X\times Y})\\
         &= \kappa(V, K_{V}-g^{*}K_{Y}) + \dim X > 0. 
    \end{align*}
    Suppose that $g$ is trivial. 
    Then for general $y\in Y$, $W_{y} := (\mathrm{pr}_{Y}\circ f)^{-1}(y)$ is smooth and irreducible, and $f_{y}:= f|_{W_{y}}\colon W_{y}\rightarrow X$ be an $a$-cover by Claim \ref{claim: factorization of a-covers}. 
    Moreover, we may assume that the ramification divisor of $f_{y}$ is the transversal intersection of the ramification divisor of $f$ and $W_{y}$. 
    Thus we have $K_{W_{y}}-f_{y}^{*}K_{X} = (K_{W}-f^{*}K_{X\times Y})|_{W_{y}}$ and hence $\kappa(W, K_{W}-f^{*}K_{X\times Y}) \ge \kappa(W_{y}, K_{W_{y}}-f_{y}^{*}K_{X}) > 0$.
    
    Therefore, $f$ is not adjoint rigid in both cases, completing the proof.
\end{proof}

\section{Subvarieties with higher $a$-invariants}\label{section: subvar with higher a}

In this section, we prove Theorem \ref{classification higher a}.
We recall the following theorem.
\begin{thm}[\cite{2019}, Theorem 3.3]
    Let $X$ be a smooth uniruled projective variety with a nef and big $\mathbb{Q}$-divisor $L$.
    Then the union $V$ of all subvarieties $Y \subset X$ with $a(Y, L|_{Y}) > a(X, L)$ is a proper closed subset of $X$.
    Moreover, each irreducible component $V_{i}$ of $V$ also satisfies $a(V_{i}, L|_{V_{i}}) > a(X, L)$.
\end{thm}

First, we consider smooth Fano varieties of coindex $3$, dimension at least $5$ and Picard rank $1$.
\begin{lem}\label{lem: higher a-invariant subvar}
Let $X$ be a smooth Fano variety of coindex $3$ and dimension $n\ge 5$ with $\Pic (X) = \mathbb{Z}H$.
Then there is no subvariety $Y$ of $X$ with $a(Y, -K_{X}) > 1$.
\end{lem}
To prove lemma \ref{lem: higher a-invariant subvar}, we will use the following lemmas and a theorem.
\begin{lem}\label{lem: higher a-invar not adjoint rigid}
Let $X$ be a smooth Fano variety of coindex $3$ and dimension $n\ge 4$ with $\Pic (X) = \mathbb{Z}H$.
Then any subvariety $Y$ with $a(Y, -K_{X}) > 1$ is covered subvarieties $Y \subset X$ such that $(Y,H)$ is isomorphic to $(\mathbb{P}^{n-2}, \mathcal{O}(1))$.  %
\end{lem}
\begin{proof}
Since the index of $X$ is $n-2$, we may assume that $Y$ has codimension at most $2$.
If the codimension of $Y$ is $2$, then \cite[Proposition 1.3]{Hoering2010} shows $(Y, H)$ is isomorphic to  a projective space $(\mathbb{P}^{n-2}, \mathcal{O}(1))$.
Suppose that $Y$ is a divisor.
Cutting by a general linear space of codimension $n-3$, we obtain a smooth Fano threefold $X'$ of index $1$ and a subvariety $Y' \subset X'$ such that $a(Y', H) = a(Y, H)-n+3 > a(X', H)$ and $\kappa (Y', K_{Y'} + a(Y', H)H) = \kappa (Y, K_{Y} + a(Y, H)H)$ \cite[Theorem 3.15]{exceptional_sets}. 
Then by \cite[Theorem 4.1]{beheshti2020moduli}, we see that $\kappa (Y', K_{Y'} + a(Y', H)H) = 1$ since the Picard rank of $X'$ is $1$.
Let $\phi \colon Y \rightarrow B$ be the Iitaka fibration for $K_{Y} + a(Y, H)H$.
Then the general fiber $F$ of $\phi$ has dimension $n-2$ and $a(F, H) > n-2$, hence $(F, H)$ is isomorphic to $(\mathbb{P}^{n-2}, \mathcal{O}(1))$.
\end{proof}
\begin{thm}[\cite{Gushel}]\label{thm: gushel-mukai}
    Let $X \subset \mathbb{P}^{10}$ be a Gushel-Mukai manifold.  If $\dim X \geq 5$, $X$ does not contain any linear spaces of codimension $2$.  If $\dim X = 4$, $X$ may contain only finitely many linear spaces of codimension $2$.
\end{thm}

\begin{lem}\label{lem: genus 8}
    Let $X$ be a smooth linear section of the Grassmannian $G(2,6) \subset \mathbb{P}^{14}$ of dimension $n \ge 4$.
    Then the evaluation morphism $\mathrm{ev}\colon \overline{M}_{0,1}(X,1)\rightarrow X$ is flat of relative dimension $n-4$.
    Moreover, when $n \ge 5$, any fiber of $\mathrm{ev}$ is a smooth complete intersection of divisors on $\mathbb{P}^{1} \times \mathbb{P}^{3}$ of bidegree $(1,1)$.
\end{lem}
\begin{proof}
    Let $V$ be a $6$-dimensional vector space and let $G := \mathrm{Gr}(2,V)$ be the Grassmannian of planes on $V$. 
    Then $G$ has dimension $8$.
    The Pl\"{u}cker embedding embeds $G$ into $\mathbb{P}^{14}\cong \mathbb{P}(\bigwedge^{2}V)$ given by $\langle v, w\rangle \mapsto [v\wedge w]$.
    By \cite{Abdelkerim_Coskun2012}, the space $\overline{M}_{0,0}(G,1)$ of lines on $G$ is isomorphic to the flag variety $F(1,3,6)$; for a pair $(V_{1}, V_{3})$ such that $\dim V_{i} = i$ and $V_{1} \subset V_{3} \subset V$, the corresponding line is the pencil $\{W\in G \mid V_{1}\subset W \subset V_{3}\}$ of planes.
    This shows that for a given plane $W \in G$, the fiber of the evaluation morphism $\mathrm{ev}\colon \overline{M}_{0,1}(G,1)\rightarrow G$ over $W$ is isomorphic to $\{(V_{1}, V_{3}) \in F(1,3,6)\mid V_{1}\subset W\subset V_{3}\} \cong \mathbb{P}^{1} \times \mathbb{P}^{3}$.
    This proves our claim in dimension $8$.
    \par
    We next consider hyperplane sections of $G$.
    From now on, we fix a basis $\{e_{1}, \dots, e_{6}\}$ of $V$ and the corresponding coordinates $z_{ij}\, (1\le i < j\le 6)$ of $\mathbb{P}^{14}$.
    Let $H = V(\sum_{i<j}a_{ij}z_{ij}) \subset \mathbb{P}^{14}$ be a hyperplane.
    Now we define a  skew symmetric matrix $Q_{H} := (a_{ij})$, where $a_{ji} = -a_{ij}$.
    This correspondence $H\mapsto Q_{H}$ induces a bijection $\mathbb{P}^{*}(\bigwedge^{2}V) \rightarrow \{(a_{ij}) \in \mathrm{Mat}(V)\mid a_{ji} = -a_{ij}\}/k^{*}$.
    For a point $\langle v,w\rangle \in G$, the condition $\langle v,w\rangle \in H$ is equivalent that ${}^{t}vQ_{H}w = 0$.
    We take a point $W = \langle e_{1}, e_{2}\rangle$ and suppose that $H$ contain $W$, i.e., $a_{12} = 0$.
    Then $G\cap H$ is smooth at $W$ if and only if $T_{W}G \not\subset H$.
    Since $T_{W}G \subset \mathbb{P}^{14}$ is spanned by $9$ points $[e_{i} \wedge e_{j}]\, (i =1,2, i<j\le 6)$ by \cite{Abdelkerim_Coskun2012}, the latter condition is equivalent to $W \not\subset \operatorname{Ker} Q_{H}$.
    Hence we see that $G \cap H$ is smooth if and only $Q_{H}$ is regular since $\operatorname{Ker} Q_{H}$ is even dimensional.
    \par
    Suppose that $G\cap H$ is smooth and contains $W = \langle e_{1}, e_{2}\rangle$.
    As proved in the first paragraph, any line on $G$ passing through $W$ is spanned by two points $\langle e_{1}, e_{2}\rangle$ and $\langle c_{1}e_{1} + c_{2}e_{2}, c_{3}e_{3} + \dots + c_{6}e_{6}\rangle$ for $((c_{1}:c_{2}), (c_{3}:\dots:c_{6})) \in \mathbb{P}^{1} \times \mathbb{P}^{3}$.
    Such a line is contained in $G\cap H$ if and only if ${}^{t}(c_{1}e_{1} + c_{2}e_{2})Q_{H}(c_{3}e_{3} + \dots + c_{6}e_{6}) = 0$.
    Thus the fiber of $\overline{M}_{0,1}(G\cap H, 1)\rightarrow G\cap H$ over $W$ is isomorphic to an ample divisor $A_{H} := V({}^{t}(x_{1}e_{1} + x_{2}e_{2})Q_{H}(x_{3}e_{3} + \dots + x_{6}e_{6})) \subset \mathbb{P}^{1}\times \mathbb{P}^{3}$ of bidegree $(1,1)$.
    Hence $\overline{M}_{0,1}(G\cap H, 1)\rightarrow G\cap H$ is flat of relative dimension $3$.
    \par
    Moreover, we claim that $A_{H}$ is smooth.
    It suffices to show that the projection $A_{H} \rightarrow \mathbb{P}^{1}$ is smooth.
    Since $W = [e_{1}\wedge e_{2}] \in H$, we have $a_{12} = 0$.
    Hence the first and second row of $Q_{H}$ are of the form 
    \[
    \begin{pmatrix}
        0 & 0 & a_{13} & \cdots & a_{16}\\
        0 & 0 & a_{23} & \cdots & a_{26}\\
    \end{pmatrix}.
    \]
    Since $Q_{H}$ is regular, so is the above submatrix.
    This shows that for any $(s:t) \in \mathbb{P}^{1}$, the fiber $V(\sum_{j=3}^{6}(sa_{1j} + ta_{2j})x_{j}) \subset \mathbb{P}^{3}$ of $A_{H}\rightarrow \mathbb{P}^{1}$ over $(s:t)$ is smooth by the Jacobian criterion.
    Therefore, $A_{H}\rightarrow \mathbb{P}^{1}$ is smooth, which proves the lemma in dimension $7$.
    \par
    Finally, we consider complete intersections $X = G\cap H_{1}\cap \dots \cap H_{\ell}$ for $2\le \ell \le 4$.
    We may assume that $W = \langle e_{1}, e_{2}\rangle$ is contained in $X$.
    Then $X$ is smooth at $W$ if and only if $G\cap H_{p}$ is smooth at $W$ for any $p = (s_{1}:\dots:s_{\ell}) \in \mathbb{P}^{\ell - 1}$, where $H_{p}$ is the hyperplane corresponding to the matrix $s_{1}Q_{H_{1}} + \dots + s_{\ell}Q_{H_{\ell}}$.
    Then the fiber of $\overline{M}_{0,1}(X,1)\rightarrow X$ over $W$ is isomorphic to the complete intersection $A_{H_{1}}\cap \dots \cap A_{H_{\ell}} \subset \mathbb{P}^{1}\times \mathbb{P}^{3}$.
    Moreover, when $\ell = 2,3$, one can prove that $A_{H_{1}}\cap \dots \cap A_{H_{\ell}}$ is smooth as discussed above, which completes the proof.
\end{proof}

\begin{proof}[Proof of Lemma \ref{lem: higher a-invariant subvar}]
By Lemma \ref{lem: higher a-invar not adjoint rigid}, it suffices to show that $X$ does not contain a $\mathbb{P}^{n-2}$.
So, we may assume that $Y \cong \mathbb{P}^{n-2}$.
We prove this based on the classification (Theorem \ref{classification Picard rank 1}).
\begin{enumerate}[(1)]
\item Suppose the fundamental divisor $H$ is not very ample. 
There are two possibilities:
    \begin{description}
        \item[($g=2$)]
        $X$ admits a double cover $f\colon X \rightarrow \mathbb{P}^{n}$ branched over $B \in |\mathcal{O}_{\mathbb{P}^{n}}(6)|$.\\
        Let $Z = f(Y)$ be the image of $Y$.
        Then $Z \subset \mathbb{P}^{n}$ is also an $(n-2)$-space such that $Z\cap B$ is non-reduced.
        So, we may write $Z = V(x_{0}, x_{1})$ and $B = V(x_{0}f_{0} + x_{1}f_{1} + g^{2})$, where $f_{0}, f_{1} \in k[x_{0}, \dots, x_{n}]_{5}$, $g \in k[x_{0}, \dots, x_{n}]_{3}$.
        Then the singular locus of $X$ must contain $V(x_{0}, x_{1}, f_{0}, f_{1}, g)$, which is non-empty, a contradiction.

        \item[($g=3$)]
        $X$ admits a double cover $f\colon X \rightarrow Q^{n}$ branched over $B \in |\mathcal{O}_{Q^{n}}(4)|$, where $Q^{n} \subset \mathbb{P}^{n+1}$ is a smooth quadric.\\
        Let $Z = f(Y)$ be the image of $Y$.
        Then $Z$ is also an $(n-2)$-space contained in $Q^{n}$. 
        This is a contradiction since $Q^{n}$ cannot contain $\mathbb{P}^{n-2}$ for $n\ge 5$.
    \end{description}
            
\item Suppose $H$ is very ample and $g\le 5$. 
Then $X$ is realized as a complete intersection of hypersurfaces of degree $d_{1}, \dots, d_{g-2}$ in a projective space $\mathbb{P}^{n+g-2}$. 
By Lefschetz hyperplane theorem, we have an isomorphism $H^{4}(\mathbb{P}^{n+g-2}, \mathbb{Z})\rightarrow H^{4}(X, \mathbb{Z})$. 
Since $\dim H^{4}(\mathbb{P}^{n+g-2}, \mathbb{Z}) = 1$, we see that the degree of codimension $2$ subvarieties of $X$ is divisible by $d_{1}\cdots d_{g-2} > 1$. 
Hence $X$ cannot contain codimension $2$ linear spaces.

\item Suppose that $g\ge 6$. There are $5$ possibilities: 
\begin{description}
    \item[($g=6$)] $X$ is a Gushel-Mukai manifold.
    In this case, the claim follows from Theorem \ref{thm: gushel-mukai}.

    \item[($g=7$)] $X$ is a linear section of the orthogonal Grassmannian $OG_{+}(5, 10) \subset \mathbb{P}^{15}$ embedded by the half-spinor coordinates. 
    Since $\dim X \ge 5$, we obtain an isomorphism $H^{4}(OG_{+}(5, 10), \mathbb{Z}) \rightarrow H^{4}(X, \mathbb{Z})$ by Lefschetz hyperplane theorem.
    By \cite{Bernstein1973}, we see that $\dim H^{4}(OG_{+}(5, 10)) = 1$.
    Moreover, $OG_{+}(5, 10)$ does not contain codimension $2$ linear space since it is homogeneous.
    Therefore, $X$ does not contain codimension $2$ linear space.

    \item[($g=8$)] 
    $X$ is a smooth linear section of the Grassmannian $G(2,6) \subset \mathbb{P}^{14}$ in the Pl\"{u}cker embedding.
    In this case, the claim follows from Lemma \ref{lem: genus 8}.
    Indeed, if $X$ contains $Y\cong \mathbb{P}^{n-2}$, then the fiber of $\mathrm{ev}\colon \overline{M}_{0,1}(X,1)\rightarrow X$ over any point in $Y$ must contain $\mathbb{P}^{n-3}$.
    However, it is impossible by Lemma \ref{lem: genus 8}.

    \item[($g=9$)] $X$ is a linear section of the Lagrangian Grassmannian $LG(3, 6) \subset \mathbb{P}^{19}$ in the Pl\"{u}cker embedding.
    The claim follows by the same argument as in the case of $g = 7$.

    \item[($g=10$)] $X$ is a closed orbit $G_2/P \subset \mathbb{P}^{13}$ of the adjoint representation of $G_2$, which is in particular homogeneous.  Every rational curve on $X$ is therefore free.  This implies $X$ cannot contain a $\mathbb{P}^3$.  \qedhere  
\end{description}
\end{enumerate}
\end{proof}

Next, we consider smooth Fano $4$-folds of coindex $3$ and Picard rank $1$.
\begin{lem}\label{lem: no planes on general fourfolds}
Let $X$ be a general smooth Fano $4$-fold of coindex $3$ with $\mathrm{Pic}(X) = \mathbb{Z}H$.  There are no subvarieties $Y \subset X$ with $(Y,H) \cong (\mathbb{P}^2, \mathcal{O}(1))$ %
\end{lem}

\begin{proof}
Let $Z$ be a smooth Fano $5$-fold of genus $g$, and let $H$ likewise denote the fundamental divisor class on $Z$.  Consider the natural map $\phi \colon Z \rightarrow \mathbb{P}^{g+3}$ induced by the complete linear system $|H|$.  We may realize $X$ as the $\phi$-preimage of a general hyperplane.  

Suppose to the contrary that there exists a subvariety $Y \subset X$ such that $(Y,H) \cong (\mathbb{P}^2, \mathcal{O}(1))$.  The image of $Y$ under $\phi$ would be a plane $\mathbb{P}^2 \subset \mathbb{P}^{g+3}$.  Set 
\[\Sigma = \{(P, H) \mid P\subset H\} \subset \mathbb{G}(2, g+3) \times |\mathcal{O}_{\mathbb{P}^{g+3}}(1)|.\]
Consider the incidence correspondence
\[
  \xymatrix{
  	\Sigma\ar[d]_-{p_{1}}\ar[r]^-{p_{2}} & |\mathcal{O}_{\mathbb{P}^{g+3}}(1)|.\\
    \mathbb{G}(2, g+3) & 
  }
\]
In this setting, we have 
\begin{enumerate}
    \item $\dim |\mathcal{O}_{\mathbb{P}^{g+3}}(1)| = g + 3$,
    \item $p_{2}$ is a flat morphism whose fiber is isomorphic to $\mathbb{G}(2, g + 2)$, hence $p_{2}$ has relative dimension $3g$.
    \item $\dim \mathbb{G}(2, g+3) = 3(g + 1)$,
    \item By the above equalities, $p_{1}$ is a flat morphism of relative dimension $g$.
\end{enumerate}
Assume that the general hyperplane section of $Z$ contains a subvariety mapping isomorphically to a plane under $\phi$.  One can see that there exists a $3$-dimensional subset $V\subset \mathbb{G}(2, g+3)$ such that any plane $P \in V$ is the isomorphic image of a subvariety in some smooth hyperplane section $Z\cap H$.
This implies that rational curves of degree $d\ge 2$ contained in some $P\in V$ form a component $N$ of $\overline{M}_{0,0}(Z,d)$, since $\dim \overline{M}_{0,0}(Z,d) = 3d + 2$ and $\dim \overline{M}_{0,0}(\mathbb{P}^{2}, d) = 3d - 1$.
Let $d = 3$, so that $N$ parameterizes planar cubic curves. Such curves are not embedded.
On the other hand, %
the general map $f \colon \mathbb{P}^1 \rightarrow Z$ parameterized by $N$ is free by Lemma \ref{lem: higher a-invariant subvar} and Proposition \ref{prop: nondominant curves subvariety a-invar}. The vector bundle $f^*\mathcal{T}_Z$ has at least three ample summands because $Y \subset X$ and both $f^*\mathcal{T}_Y$ and $f^*N_{X/Z}$ are ample.  However, by \cite[Theorem 3.14]{Kollar1996}, this implies $N$ generically parameterizes embedded curves, a contradiction.
\end{proof}

\begin{rem}
    Some smooth quartic $4$-folds contain a $\mathbb{P}^{2}$.
    Let $X = V(x_{0}f + x_{1}g + x_{2}h) \subset \mathbb{P}^{5}$. 
    If the polynomials $f, g, h$ are general, then $X$ is smooth.
    Moreover, $X$ contains $\mathbb{P}^{2} = V(x_{0}, x_{1}, x_{2})$ and no other planes.  The following lemma shows there are at most finitely many planes in any Fano $4$-fold of coindex $3$.
\end{rem}

\begin{lem}\label{lem: finitely many planes}
    Let $X$ be an arbitrary smooth Fano $4$-fold of coindex $3$ with $\Pic(X) = \mathbb{Z}H$.  There are finitely many subvarieties $Y \subset X$ with $(Y,H) \cong (\mathbb{P}^2, \mathcal{O}(1))$. %
\end{lem}
\begin{proof}
    First, suppose $X$ has genus at most $5$.  In this case, we may realize $X$ as a smooth complete intersection in some weighted projective space $\mathbb{P}$.  Suppose there exists a subvariety $Y \cong \mathbb{P}^2 \subset X$.  Consider the exact sequence
    $$ 0 \rightarrow N_{Y/X} \rightarrow N_{Y/\mathbb{P}} \rightarrow N_{X /\mathbb{P}}|_Y \rightarrow 0.$$
    Since $X$ is a complete intersection of very ample divisors on $\mathbb{P}$, $N_{X /\mathbb{P}}|_Y$ is a direct sum of line bundles.  As the first chern class of $N_{Y/X}$ is $\mathcal{O}_{Y}(-1)$, $N_{Y/X}$ has a global section if and only if $N_{Y/X}^\vee(-1) \cong N_{Y/X}$ has a global section.  However, $N_{Y/X}^\vee(-1)$ fits into an exact sequence
    $$ 0 \rightarrow N_{X /\mathbb{P}}^\vee|_Y(-1) \rightarrow N_{Y/\mathbb{P}}^\vee(-1) \rightarrow N_{Y/X}^\vee(-1)  \rightarrow 0.$$
    Since $N_{X /\mathbb{P}}^\vee|_Y(-1)$ is a direct sum of line bundles, $h^1(Y, N_{X /\mathbb{P}}^\vee|_Y(-1)) = 0$.  Similar reasoning shows $H^0(Y,N_{Y/\mathbb{P}}^\vee(-1)) = 0$ as well.  It follows that $H^0(Y,N_{Y/X}) = 0$.  As this is the tangent space to the Hilbert scheme of planes in $X$ at the point $[Y]$, there may be only finitely many planes in $X$.

    Theorem \ref{thm: gushel-mukai} proves our claim when the genus of $X$ is 6.  Suppose $g(X) = 7$ instead.  Here, we use an argument shared with us by Kuznetsov.  Let $\mathbb{G} = OG_{+}(5, 10) \subset \mathbb{P}^{15}$, so that $X$ is the intersection of $\mathbb{G}$ with a linear subspace.  Any plane $Y \cong \mathbb{P}^2 \subset X$ is contained in a $4$-dimensional linear space $L \cong \mathbb{P}^4 \subset \mathbb{G}$.  Since $Y$ is the intersection of $L$ with the baselocus $\tilde{X} \subset \mathbb{G}$ of a pencil of hyperplanes containing $X$, there is an exact sequence
    $$0 \rightarrow N_{Y/X} \rightarrow N_{Y/\tilde{X}} \rightarrow N_{X/\tilde{X}}|_Y \rightarrow 0.$$
    We have $N_{Y/\tilde{X}} \cong N_{L/\mathbb{G}}|_Y \cong \bigwedge^2(T_{L}(-1))|_Y$.  Moreover, $N_{X/\tilde{X}}|_Y \cong \mathcal{O}_Y(1)^{\oplus 4}$.  As before, this shows $N_{Y/X}^\vee(-1) \cong N_{Y/X}$ has no global sections.

    When the genus of $X$ is $8$, Lemma \ref{lem: genus 8} shows that $X$ contains no linear spaces of codimension $2$.
    Indeed, if $Y\cong \mathbb{P}^{2} \subset X$, then the fiber of $\mathrm{ev}\colon \overline{M}_{0,1}(X,1)\rightarrow X$ over any point in $Y$ should contain a line, which is a contradiction.

    Suppose $X$ is s linear section of the Lagrangian Grassmannian $LG(3, 6) \subset \mathbb{P}^{19}$ in the Pl\"{u}cker embedding.  In this case, $LG(3,6)$ contains no planes \cite[Lemma~2.5.1]{lagrangianGrass}; %
    hence, neither does $X$.  Similarly, if $X \subset G_2/P \subset \mathbb{P}^{13}$ is a $4$-fold of genus $10$, then $X$ cannot contain any planes as $G_2/P$ contains no planes.  This latter claim follows by homogeneity of $G_2/P$: if $G_2/P$ contained a plane, then by homogeneity every point in $G_2/P$ would lie in a plane.  A base change of this family of planes would be an $a$-cover of $G_2/P$; however, by Theorem \ref{classification a-covers} there are no $a$-covers of $G_2/P$.  Theorem \ref{classification a-covers} for $G_2/P$ follows from \cite{Thomsen1998} and \cite{KimPandharipande2001}.  %
\end{proof}

Finally, we consider coindex $3$ Fano varieties of Picard rank at least $2$.

\begin{lem}\label{lem: Picard rank 2 subvarieties higher a-invar}
    Let $X$ be a smooth coindex $3$ Fano variety of dimension $n \ge 4$ and Picard rank at least $2$.  Let $H$ be the fundamental divisor on $X$.  Suppose that %
    $Y \subset X$ is a subvariety with $a(Y, -K_X|_Y) > 1$.  Then either
    \begin{enumerate}
        \item $Y$ is a contractible divisor on $X$, or
        \item $n = 4$, $(Y,H) \cong (\mathbb{P}^{n-2},\mathcal{O}(1))$ and $Y$ is contracted to a point by an elementary fiber-type contraction on $X$. 
    \end{enumerate}
    Furthermore, when $g(X) \in \{7, 12, 13, 16, 21\}$ and $X$ is general in moduli, there are no subvarieties of type (2).
\end{lem}

\begin{proof}
    Let $H$ be the fundamental divisor on $X$.  Suppose $Y \subset X$ is a subvariety with $a(Y, H) > n - 2$.  Since $|H|$ is very ample, \cite[Proposition 1.3]{Hoering2010} describes all possible isomorphism types of the pair $(Y,H)$.  In particular, \cite[Proposition 1.3]{Hoering2010} implies that $Y$ is covered by $H$-lines.  The class of these $H$-lines spans an extremal ray $R$ of $\NE(X)$.  Consider the elementary contraction $\phi : X \rightarrow B$ associated to $R$.  

    We claim that if $\dim Y = n -1$, then $Y$ must be a contractible divisor.  Otherwise, $\phi$ would be of fiber type.  In this case, \cite[Proposition 1.3]{Hoering2010} and Theorem \ref{classification Picard rank 2} show $Y$ must be a union of reducible fibers of $\phi$ or of fibers of larger than expected dimension.  However, since $Y$ is a divisor, the relative Picard rank of $\phi$ would be at least $2$, a contradiction.

    If $Y \subset X$ is not a divisor, then $(Y,H) \cong (\mathbb{P}^{n-2}, \mathcal{O}(1))$ by \cite[Proposition 1.3]{Hoering2010}.  It remains to show that subvarieties of type (2) do not appear if $n > 4$.  By Theorem \ref{classification Picard rank 2}, we need only consider smooth fundamental divisors $X \subset \mathbb{P}^3 \times \mathbb{P}^3$.  If the projection of $X$ onto either $\mathbb{P}^3$-factor had a $3$-dimensional fiber, then $X$ would be singular.  Hence, $X$ does not contain any subvarieties of higher $a$-invariant.

    Lastly, we show that if $g(X) \in \{7, 12, 13, 16, 21\}$ and $X$ is general in moduli, there are no subvarieties of type (2).  This claim is trivial if $g(X) \in \{13,16,21\}$, as each fiber type contraction of $X$ has irreducible equidimensional fibers.  

    Suppose $g(X) = 7$; that is, $X$ is a double cover of $\mathbb{P}^2 \times \mathbb{P}^2$ branched over a divisor $B$ of bidegree $(2,2)$.  If $X$ contains a subvariety $Y$ with $a(Y, H) > n - 2$, then at least one fibration $\phi : X \rightarrow \mathbb{P}^2$ must have reducible fibers.  In this case, the $\phi$-fiber of $B$ over some point $p \in \mathbb{P}^2$ must be a non-reduced line $B_p$.  %
    There is a rational map
    \[\mathbb{P}H^0(\mathcal{O}_{\mathbb{P}^2 \times \mathbb{P}^2}(2,2)) \times \mathbb{P}^2 \dashrightarrow \mathbb{P}H^0(\mathcal{O}_{\mathbb{P}^2}(2))\]
    sending a global section of $\mathcal{O}_{\mathbb{P}^2 \times \mathbb{P}^2}(2,2)$ to its restriction to the fiber over a point in $\mathbb{P}^2$.  This map has equidimensional fibers.  Hence, as the locus of doubled lines is of codimension $3$ in the target, for a fixed general branch divisor $B \in \mathbb{P}H^0(\mathcal{O}_{\mathbb{P}^2 \times \mathbb{P}^2}(2,2))$, there are no reducible fibers. %

    Suppose $g(X) = 12$; that is, $X$ is the blow-up of a smooth four dimensional quadric along a conic $C$ not lying in a plane lying in $X$.  Let $\pi : X \rightarrow \mathbb{P}^2$ be the projection of $X$ from the plane containing $C$.  Since $C$ is irreducible, each fiber of $\pi$ is irreducible as well.  Hence, the exceptional divisor on $X$ is the only subvariety of higher $a$-invariant.
\end{proof}

The codimension $2$ subvarieties of higher $a$-invariant are explicitly calculable.  We provide the following as an example.

\begin{exmp}
    Any smooth $(1,1)$ divisor in $\mathbb{P}^2 \times Q^3$ has exactly two fibers of larger than expected dimension over $Q^3$, and no other subvarieties of higher $a$-invariant.  Indeed, the $(1,1)$ divisor $V(x_0f + x_1g + x_2h)$ does not have any reducible fibers over $\mathbb{P}^2$, as $Q^3$ does not contain any linear spaces of codimension $1$.  The fiber over $\mathbb{P}^2$ jumps in dimension over $Q^3 \cap V(f,g,h)$, which is the intersection of $Q^3$ with a line.  If the line $V(f,g,h)$ were contained in $Q^3$, then the Jacobian of the divisor vanishes along a section of $\mathbb{P}^2 \times V(f,g,h) \rightarrow V(f,g,h)$, considered as a subscheme of $\mathbb{P}^2 \times Q^3$.  Indeed, there is a relation $R_p$ among $f,g,h$ in the local ring of any point $p \in V(f,g,h) \subset Q^3$.  Over a single point in $\mathbb{P}^2$, $x_0f + x_1g + x_2h \equiv R_p$, and the intersection is singular at such a point.
\end{exmp}

Theorem \ref{classification higher a} follows from Lemma \ref{lem: higher a-invariant subvar}, \ref{lem: no planes on general fourfolds}, \ref{lem: finitely many planes}, and \ref{lem: Picard rank 2 subvarieties higher a-invar}.

\section{Spaces of Low Degree Rational Curves}\label{section: low degree curves}
In this section, we prove Theorem \ref{thm: low degree curves}.  This will be used in Sections \ref{section:mbb} and \ref{section:a covers} to prove Movable bend-and-break (Theorem \ref{MBB}) and classification of $a$-covers of $X$ (Theorem \ref{classification a-covers}).

\begin{thm}\label{thm: low degree curves}
    Let $X$ be a smooth Fano variety of coindex $3$ and dimension at least $4$.  Assume $X$ is not a Fano $4$-fold of product type.  Let $H$ be the fundamental line bundle on $X$.  Suppose $X$ is general in moduli and let $\alpha \in \Nef_1(X)_\mathbb{Z}$.
    \begin{enumerate}
        \item If $ H \cdot \alpha = 1$, $\overline{M}_{0,0}(X, \alpha)$ is irreducible and smooth.
        \item If  $H \cdot \alpha = 2$, $\alpha \notin \partial\NE(X)$, and $\dim X \geq 5$, $\overline{M}_{0,0}(X, \alpha)$ is irreducible; if  $\dim X = 4$ instead, all but one component of $\overline{M}_{0,0}(X, \alpha)$ parameterize double covers of $H$-lines.
        \item If  $H \cdot \alpha = 3$, $\alpha \notin \partial\NE(X)$, and $\dim X = 4$, there exists precisely one component of $\overline{M}_{0,0}(X, \alpha)$ which generically parameterizes embedded curves.
    \end{enumerate}
\end{thm}

\begin{rem}
We will later extend parts (2) and (3) of Theorem \ref{thm: low degree curves} to arbitrary coindex 3 Fano varieties.  We will also prove $\overline{M}_{0,0}(X,\alpha)$ is irreducible when $H \cdot \alpha = 1$ and $\dim X \geq 5$.  Each of these extensions requires Theorem \ref{classification a-covers} for general $X$ and Lemma \ref{lem: specialize Manin components}.  %
For arbitrary $X$, the proofs of Theorem \ref{classification a-covers} and results in Section \ref{section: intro} require these extensions.  %
\end{rem}

We split this study into two cases, depending on the genus $g$ of $X$.  When $g \leq 11$, we compare the space $\overline{M}_{0,0}(X,\alpha)$ with the space $\overline{M}_{0,0}(\mathbb{G}, \alpha)$ for some ambient variety $\mathbb{G}$, in which we embed $X$ as a complete intersection.  When $g > 11$, we use the action of $\mathrm{Aut}(X)$ to study components of $\overline{M}_{0,0}(X,\alpha)$ for all $\alpha$.  This latter case includes all $X$ which contain contractible divisors.

\subsection{Rational curves on varieties with $g(X) \le 11$}\label{subsection: genus at most 11}
Below, we prove Theorem \ref{thm: low degree curves} in pieces.  
First, we focus on coindex $3$ Fano varieties $X$ with $g(X) \le 11$.
Our strategy is to relate irreducibility of a space of curves on $X$ to irreducibility of a space of curves on some ambient variety $\mathbb{G}$ which contains deformations of $X$ as complete intersections.

\begin{constr}\label{families of Fano varieties}
Let $X$ be a smooth coindex $3$ Fano variety of dimension $n \ge 4$ and genus $g(X) \le 11$.
By Theorems \ref{classification Picard rank 1} and \ref{classification Picard rank 2}, we may realize $X$ as a complete intersection of divisors $A_1, \dots, A_m$ inside a projective variety $\mathbb{G}$, which contains at most one singular point. 
We let
$$\mathbf{M} \subset \overline{\mathbf{M}} := |\mathcal{O}_\mathbb{G}(A_1)| \times \dots \times |\mathcal{O}_\mathbb{G}(A_m)|$$
be the open sublocus parameterizing intersections which avoid the singularities of $\mathbb{G}$. 
The varieties $\mathbb{G}$ and $\mathbf{M}$ are described as follows. 
When $\rho(X) = 1$ and $g(X) \le 6$, $\mathbb{G}$ and $\overline{\mathbf{M}}$ are as in the table below.
\[\begin{array}{c|rr}
    g(X) & \multicolumn{1}{c}{\mathbb{G}} & \multicolumn{1}{c}{\overline{\mathbf{M}}} \\ \hline
    2 & \mathbb{P}(1^{n+1}, 3) & |\mathcal{O}_{\mathbb{G}}(6)| \\
    3 & \mathbb{P}(1^{n+2}, 2) & |\mathcal{O}_{\mathbb{G}}(2)| \times |\mathcal{O}_{\mathbb{G}}(4)| \\
    4 & \mathbb{P}^{n+2} & |\mathcal{O}_{\mathbb{G}}(2)| \times |\mathcal{O}_{\mathbb{G}}(3)| \\
    5 & \mathbb{P}^{n+3} & |\mathcal{O}_{\mathbb{G}}(2)| \times |\mathcal{O}_{\mathbb{G}}(2)| \times |\mathcal{O}_{\mathbb{G}}(2)| \\
    6 & \mbox{cone over }G(2,5) & |\mathcal{O}_{\mathbb{G}}(2)|\times |\mathcal{O}_{\mathbb{G}}(1)|^{6-n}
\end{array}\]
When $\rho(X) = 1$ and $g(X) \ge 7$, $\mathbb{G}$ is the maximal homogeneous variety and each $A_{i}$ is a fundamental divisor on $\mathbb{G}$.
When $(\rho(X), g(X)) = (2, 7)$, $\mathbb{G}$ is the cone over the Segre embedding of $\mathbb{P}^{2}\times \mathbb{P}^{2}$, and $\overline{\mathbf{M}} = |\mathcal{O}_{\mathbb{G}}(2)|$.
When $(\rho(X), g(X)) = (2, 9)$, $\mathbb{G} = \mathbb{P}^{2}\times \mathbb{P}^{3}$ and $\mathbf{M} = |\mathcal{O}_{\mathbb{G}}(1,2)|$.
When $(\rho(X), g(X)) = (2, 11)$, $\mathbb{G}$ is equal to $\mathbb{P}^{3}\times \mathbb{P}^{3}$ or $\mathbb{P}^{2}\times Q^{3}$, and each $A_{i}$ is a fundamental divisor on $\mathbb{G}$.

Let $\mathbb{G}^o \subset \mathbb{G}$ be the smooth locus of $\mathbb{G}$.  We %
consider $X$ as a fiber of the family 
$$\mathcal{X} = \{ (A_1, \dots, A_m, x)  \in \mathbf{M} \times \mathbb{G}^o \mid x \in A_1 \cap \dots \cap A_m\}.$$
Let $\pi \colon \mathcal{X} \rightarrow \mathbf{M}$ and $\mathrm{ev}\colon \mathcal{X} \rightarrow \mathbb{G}$ be the projections.  We call the locus $D \subset \mathbf{M}$ where $\pi \colon \mathcal{X} \rightarrow \mathbf{M}$ fails to be smooth the \textit{discriminant locus}.
\end{constr}

\begin{lem}
The space $\mathbf{M}$ constructed above is smooth and simply connected.  The discriminant locus $D \subset \mathbf{M}$ is irreducible.  %

\end{lem}
\begin{proof}
Because $\mathbb{G}$ is singular along at most one point $p \in \mathbb{G}$, $\mathbf{M} \subset \overline{\mathbf{M}} = |\mathcal{O}_\mathbb{G}(A_1)| \times \dots \times |\mathcal{O}_\mathbb{G}(A_n)|$ is either the entire space or a complement of $|\mathcal{O}_\mathbb{G}(A_1 - p)| \times \dots \times |\mathcal{O}_\mathbb{G}(A_n - p)|$.  Hence, $\mathbf{M}$ is smooth and simply connected because these properties hold for $\overline{\mathbf{M}}$.  

Next we show that the discriminant locus $D \subset \mathbf{M}$ is irreducible.  Note that for any smooth point $x \in \mathbb{G}$, the sublocus of $\mathbf{M}$ parameterizing intersections with larger than expected tangent space at $x$ correspond to choices of hypersurface sections containing $x$ whose derivatives are linearly dependent in $\mathcal{T}_x\mathbb{G}$.  As this locus is irreducible and of the expected dimension, varying $x$ shows $D$ is irreducible.
\end{proof}

In what follows, let $X$ be a general complete intersection parameterized by $\mathbf{M}$.  Let $H$ be the fundamental divisor on $X$, and identify $H$ with the unique corresponding divisor class on $\mathbb{G}$.  Note that $\NE(X) \cong \NE(\mathbb{G}^o)$ is independent of $X$.  We will denote this subcone of $\NE(\mathcal{X})$ by $\NE(\pi)$.  For $\alpha \in \NE(X)$, we relate irreducibility of $\overline{M}_{0,0}(X, \alpha)$ to irreducibility of $\overline{M}_{0,0}(\mathcal{X}, \alpha)$ by studying the general singular degeneration of $X$.

\begin{lem}\label{lem: terminal singularity}
    Let $\mathcal{X}$, $\mathbb{G}$ and $D\subset \mathbf{M}$ be as in Construction \ref{families of Fano varieties}.
    Suppose $D\subset \mathbf{M}$ has codimension $1$.  Then the fiber of $\pi \colon \mathcal{X} \rightarrow \mathbf{M}$ over a general point $p \in D$ is a factorial Gorenstein terminal Fano variety $\mathcal{X}_p$ of the expected dimension whose singular locus is a point $q$.  %
    The general $3$-fold linear section of $\mathcal{X}_p$ containing $q$ has a simple double point at $q$.
\end{lem}

\begin{proof}
    Let $\mathcal{X}_p$ be the fiber of $\mathcal{X}$ over a general point $p \in D$.  Recall that $\mathcal{X}_p$ is contained inside the smooth locus of $\mathbb{G}$.  By construction, $\mathcal{X}_p \subset \mathbb{G}$ is a complete intersection, and thus Gorenstein.  Provided the singular locus of $\mathcal{X}_p$ is of dimension zero, \cite[Theorem~1]{RavindraSrinivas2006} proves $\mathcal{X}_p$ is factorial, as a general fundamental divisor on $\mathcal{X}_p$ would be a smooth coindex three Fano variety.  Thus, it suffices to show $\mathcal{X}_p$ is terminal and has a unique singularity.  %

    When the fundamental linear series on $\mathcal{X}_p$ is not very ample, this claim follows immediately from Theorem \ref{classification Picard rank 1}(1). Indeed, the singular locus of a general singular weighted hypersurface lying in the smooth locus of $\mathbb{G}$ is a simple double point by the proof of \cite[Proposition~7.1(b)]{eisenbud20163264}.  An analogous constructive argument shows there are singular double covers of smooth quadrics whose singular locus is a simple double point.  %

    Suppose the fundamental linear series $|H|$ on  $\mathcal{X}_p$ is very ample instead.  Note that $H$ is also the hyperplane class on $\mathbb{G}$. %
    First, we will show the universal singularity
    $$\Sigma = \{x \in \mathcal{X} \mid d\pi_x \colon \mathcal{T}_{x} \mathcal{X} \rightarrow \mathcal{T}_{\pi(x)} \mathbf{M} \text{ is not surjective}\}$$
    has codimension $n+1$ in $\mathcal{X}$.  
    Recall that the sublocus of $\mathbf{M}$ parameterizing intersections with larger than expected tangent space at a smooth point $x \in \mathbb{G}$ is the locus where the Jacobian matrix $A : \mathcal{T}_x\mathbb{G} \rightarrow \mathbb{C}^{\dim \mathbb{G} - n}$ has non-full rank.
    This is a codimension $n+1$ locus in the fiber of $\mathrm{ev} \colon \mathcal{X} \rightarrow \mathbb{G}$, whence our claim follows.
    
    Since both $\mathcal{X}$ and $\mathbf{M}$ are smooth, $\Sigma$ contains $\mathcal{X}_{\pi(x)}$ when the latter is of dimension greater than $n$; otherwise, $\Sigma$ meets $\mathcal{X}_{\pi(p)}$ along its singular locus.  Hence, as the relative dimension of $\pi$ is $n$, if $D \subset \mathbf{M}$ has codimension one, $\mathcal{X}_p$ will have dimension $n$ and a zero-dimensional singular locus.

    Next, we will show the singular locus of $\mathcal{X}_p$ is a linear subspace.  Recall that $\mathcal{X}_p \subset \mathbb{G}$ is an  complete intersection of divisors $A_1, \ldots, A_m$.  We claim that unless $m = 1$, each divisor $A_i \subset \mathbb{G}$ is smooth by generality of $p \in D$.  Indeed, the locus of singular $A_i$ in $|\mathcal{O}_\mathbb{G}(A_i)|$ has codimension $1$ greater than the general dimension of singularity.  %
    This implies that the locus in $\mathbf{M}$ where $A_i$ is singular and the intersection $A_1 \cap \dots \cap A_m$ meets $\mathrm{sing}(A_i)$ has codimension $m$. Since $D \subset \mathbf{M}$ has codimension $1$, it follows  for general $p \in D$, each corresponding divisor $A_i$ is smooth when $m > 1$.  In particular, we may assume $\mathcal{X}_p$ is a hyperplane section of a smooth variety, or a variety $\mathbb{G}$ with one singular point, and apply Zak's Theorem on Tangencies (e.g., \cite[Theorem 3.4.17]{Lazarsfeld2004positivityI}) to conclude the singular locus of $\mathcal{X}_p$ is a linear subspace.  In other words, $\mathcal{X}_p$ has a unique singular point.  %

    Let $q \in \mathcal{X}_p$ be the singular point.  The general threefold linear section $Y \subset \mathcal{X}_p$ containing $q$ has a simple double point at $q$.  Indeed, by generality of $\mathcal{X}_p$, $Y$ is also a general singular complete intersection.  When $\rho(\mathcal{X}_p) = 1$, it follows from \cite[Theorem~1.6, Remark~7.8]{kuznetsov2023one} that $q \in Y$ is a simple double point.  Otherwise, when $\rho(\mathcal{X}_p) = 2$, this claim may be observed from the defining equations of $\mathcal{X}_p \subset \mathbb{G}$ in local coordinates.

    Lastly, we will show $\mathcal{X}_p$ is terminal.  %
    By preceding arguments, we may assume $\mathcal{X}_p$ is a hyperplane section of a variety $Z$ which is smooth along $q$.  Since $Y \subset \mathcal{X}_p$ has a simple double point at $q$, the exceptional divisor $E_q$ of the blow-up $\text{Bl}_q \mathcal{X}_p \rightarrow \mathcal{X}_p$ is a quadric of rank at least four in the exceptional divisor of $\text{Bl}_q Z \rightarrow Z$.      %
    Let 
    $$\widetilde{\mathcal{X}}_p \xrightarrow{\phi} \text{Bl}_q \mathcal{X}_p \xrightarrow{\psi} \mathcal{X}_p$$ 
    be a resolution of singularities with simple normal crossings exceptional locus.  As the multiplicity of the defining equation of $\mathcal{X}_p \subset Z$ is at most two along the image of each exceptional divisor, the discrepancy of any exceptional divisor is positive.
\end{proof}

This description of the general singular fiber shows the $\pi$-fiber of a component of $\overline{M}_{0,0}(\mathcal{X}, \alpha)$ over the general point in the discriminant locus cannot be too pathological.  We leverage this to prove the following lemma.

\begin{lem}\label{lem: irred fiber}
    Let $\mathcal{X}$, $\mathbb{G}$ and $D\subset \mathbf{M}$ be as in Construction \ref{families of Fano varieties}.
    For $\alpha \in \NE(\pi)$ with $H\cdot \alpha \leq 2$, or with $H\cdot \alpha \leq 3$ if $\dim X = 4$, each component of $\overline{M}_{0,0}(\mathcal{X}, \alpha)$ which generically parameterizes embedded free curves has an irreducible general $\pi$-fiber. 
\end{lem}

\begin{proof}
    Let $Z_\alpha  \subset \overline{M}_{0,0}(\mathcal{X}, \alpha)$ be the closure of a component of the locus parameterizing embedded curves of class $\alpha$.  Consider the map $\pi_* \colon \overline{M}_{0,0}(X,\alpha) \rightarrow \mathbf{M}$. We claim that the restriction $\pi_* \colon Z_\alpha \rightarrow \mathbf{M}$ has irreducible general fiber for each $\alpha$.

    To derive a contradiction, let $\tilde{Z}$ be a resolution of $Z_\alpha$ and assume the Stein factorization of $\pi_* \colon \tilde{Z} \rightarrow \mathbf{M}$ is nontrivial.  Since $\mathbf{M}$ is smooth and simply connected, the branch locus $B \subset \mathbf{M}$ of the finite part $f \colon \tilde{\mathbf{M}} \rightarrow \mathbf{M}$ of the Stein factorization of $\pi_*$ must have codimension $1$.  Over any point $b \in B$, the fiber of $\tilde{Z}$ has an everywhere non-reduced connected component $N$.  By upper semicontinuity of fiber dimension applied to a general point of the base change of $\mathcal{X}$ along $f \colon \tilde{\mathbf{M}} \rightarrow \mathbf{M}$, some irreducible component $N_1 \subset N$ parameterizes a family of curves which dominates $\mathcal{X}_b$ and has at least the expected dimension.  We will show that for general $b \in B$, no such component $N_1$ exists.  

    Let $b \in B$ be general.  As $B \subset \mathbf{M}$ is of codimension $1$, by Lemma \ref{lem: terminal singularity} $\mathcal{X}_b$ is terminal and has at most one singular point.  Observe that $N$ is a connected union of everywhere non-reduced components of $\overline{M}_{0,0}(\mathcal{X}_b,\alpha)$.  Hence, $N_1$ must be an everywhere non-reduced, irreducible component of $\overline{M}_{0,0}(\mathcal{X}_b,\alpha)$ parameterizing a dominant family of curves of at least the expected dimension 
    $$-K_{\mathcal{X}_b} \cdot \alpha + \dim \mathcal{X}_b - 3 = (\dim \mathcal{X}_b - 2) H \cdot \alpha + \dim \mathcal{X}_b - 3.$$
    Suppose $N_1$ generically parameterized irreducible curves.  Since $N_1$ parameterizes a dominant family of curves, the general curve parameterized by $N_1$ must pass through the singularity of $\mathcal{X}_b$, as any free curve contained in the smooth locus of $\mathcal{X}_b$ would be a reduced point of $N_1$.  On the other hand, terminality of $\mathcal{X}_b$ implies any dominant family of irreducible curves of class $\alpha$ passing through the singularity of $\mathcal{X}_b$ has less than the expected dimension above.  Thus, the general curve parameterized by $N_1$ must be reducible.  

    When $H \cdot \alpha = 1$, this immediately contradicts the existence of such a component $N_1$, as $H$ is an ample Cartier divisor.  When $H \cdot \alpha \geq 2$, we use in addition that $N_1$ parameterizes the specialization of embedded free curves on a general fiber to conclude that some curve $C$ parameterized by $N_1$ passes through a general point $p \in \mathcal{X}_b$ and meets $(\dim \mathcal{X}_b - 2) H \cdot \alpha - 2$ general codimension $2$ complete intersections $Y_i \subset\mathcal{X}_b$. 

    Suppose the curve $C$ parameterized by $N_1$ is the union of an irreducible curve $C_1$ passing through a general point $p \in \mathcal{X}_b$ and a line $C_2$. Observe that the family of lines $\mathcal{F}_q$ through any point $q \in \mathcal{X}_b$ has dimension at most $\dim \mathcal{X}_b - 2$.  Thus, $C_1$ must pass through $(\dim \mathcal{X}_b - 2) H \cdot [C_1] - 2$ general codimension $2$ subvarieties $Y_i$, and $C_2$ meets $C_1$ at a point $q \in \mathcal{X}_b$ with $\dim \mathcal{F}_q = \dim \mathcal{X}_b - 2$.  As there are no infinitesimal deformations of $C_1$ preserving the intersections with $p$ and each $Y_i$, we may assume $C_1$ is general in moduli.  It follows that $\mathcal{L} = \{q \in \mathcal{X}_b \mid \dim \mathcal{F}_q = \dim \mathcal{X}_b - 2\}$ has codimension $1$ in $\mathcal{X}_b$.  However, any component $L \subset \mathcal{L}$ with codimension $1$ in $\mathcal{X}_b$ embeds as a linear subspace under the natural map $\mathcal{X}_b \rightarrow \mathbb{P}|H|$ which contradicts factoriality of $\mathcal{X}_p$.

    Suppose instead that $H \cdot \alpha = 3$, $\dim \mathcal{X}_b = 4$, and $C$ is the union of a line $C_1$ passing through a general point $p \in \mathcal{X}_b$ and a (possibly disconnected) conic $C_2$.  As there are no infinitesimal deformations of $C_1$ through $p$ and $\dim \mathcal{X}_b - 4$ general codimension $2$ subvarieties $Y_i$, $C_2$ meets at least $2\dim \mathcal{X}_b - 4$ general codimension $2$ complete intersections $Y_i$.  As before, factoriality of $\mathcal{X}_p$ shows $C_2$ must be irreducible.  It follows that $C_2$ deforms in a $3\dim \mathcal{X}_b - 6 = 6$ dimensional family covering a divisor $D$ on $\mathcal{X}_b$, since the expected dimension for deformations of $C_2$ is $5$.  This shows $a(D,H) = 3$.  Hence, by \cite[Proposition~3.17]{exceptional_sets} either $(D, H) \cong (Q, \mathcal{O}(1))$ where $Q \subset \mathbb{P}^4$ is a possibly singular quadric hypersurface, or $(D,H)$ is birational to a $\mathbb{P}^2$-bundle over a curve $Z$ and $H$ is the relative tautological line bundle.  In each case, the node $q$ of $C$ must be general in $D$, and $C$ would lie in the smooth locus of $\mathcal{X}_p$.  We see from the sequence
    $$0 \rightarrow H^0(C, N_C) \rightarrow H^0(C_1, N_C|_{C_1}) \oplus H^0(C_2, N_C|_{C_2}) \rightarrow H^0(q, N_C|_q) \rightarrow 0$$
    that $C$ deforms with dimension $h^0(C, N_C) = h^0(C_1, N_{C_1}) + h^0(C_2, N_{C_2}) - \dim \mathcal{X}_b + 2$.  Thus, $C$ would correspond to a smooth point of $N_1$, a contradiction. 
\end{proof}

\begin{lem}\label{lem: irred total space}
    Let $\mathcal{X}$, $\mathbb{G}$ and $D\subset \mathbf{M}$ be as in Construction \ref{families of Fano varieties}.
    When $H\cdot \alpha \leq 3$, the locus in $\overline{M}_{0,0}(\mathcal{X}, \alpha)$ parameterizing embedded curves is irreducible.
\end{lem}

\begin{proof}
    Each fiber $X \subset \mathcal{X}$ of $\pi$ is a complete intersection of very ample divisors inside the smooth locus $\mathbb{G}^o$ of $\mathbb{G}$.  The locus of complete intersections containing a given embedded curve of class $\alpha$ in $\mathbb{G}^o$ is irreducible and of the expected dimension, as $H\cdot \alpha \leq 3$.  Indeed, when $H$ is very ample, any embedded cubic $C \subset \mathbb{G}$ is a twisted cubic in $\mathbb{G} \subset \mathbb{P}|H|$, and hence the natural map $H^0(\mathbb{G}, \mathcal{O}(H)) \rightarrow H^0(C, \mathcal{O}(H))$ is surjective.  When $H$ is not very ample and $C \subset \mathbb{G}$ projects to a planar cubic $\overline{C} \rightarrow \mathbb{P}|H|$, then $\overline{C}$ is a complete intersection and surjectivity of relevant maps $H^0(\mathbb{G}, \mathcal{O}(kH)) \rightarrow H^0(C, \mathcal{O}(kH))$ follows from a dimension count after computing $h^0(I_{\overline{C}}(k))$.  Irreducibility of the locus in $\overline{M}_{0,0}(\mathcal{X}, \alpha)$ parameterizing embedded curves thus follows from irreducibility of the family of embedded curves of class $\alpha$ in $\mathbb{G}^o$.  This is proven in \cite{Thomsen1998}, \cite{KimPandharipande2001} when $\mathbb{G}^o = \mathbb{G}$ is homogeneous.

    Suppose $\mathbb{G}^o$ is the total space of a line bundle over a homogeneous variety.  In this case, either $\rho(X) = 1$ and $g(X) \in \{2, 3, 6\}$ or $(\rho(X), g(X)) = (2,7)$.  The space $\overline{M}_{0,0}(\mathbb{G}^o, \alpha)$ has irreducible fibers over a space of rational curves in a homogeneous variety, and is thus irreducible. %
\end{proof}

\begin{lem}
    Let $\mathcal{X}$, $\mathbb{G}$ and $D\subset \mathbf{M}$ be as in Construction \ref{families of Fano varieties}.
    When $H\cdot \alpha = 1$, $\overline{M}_{0,0}(X,\alpha)$ is smooth for general $X$.
\end{lem}

\begin{proof}
    $\overline{M}_{0,0}(X,\alpha) \subset \overline{M}_{0,0}(\mathbb{G}, \alpha)$ is the zero locus of a section of a globally generated vector bundle whose sections correspond to deformations of $X$.  Indeed, $X$ is the zero locus of a global section of a vector bundle $V \cong \mathcal{O}_\mathbb{G}(A_1)\oplus \ldots \oplus \mathcal{O}_\mathbb{G}(A_n)$ where each $A_i$ is basepoint free.  As $\mathrm{ev}\colon \overline{M}_{0,1}(\mathbb{G}, \alpha) \rightarrow \mathbb{G}$ has connected fibers, global sections of $\mathrm{ev}^* V$ may be identified with global sections of $V$.  Since $\mathrm{ev}^* V$ is ample on each fiber $C$ of $\pi \colon \overline{M}_{0,1}(\mathbb{G}, \alpha) \rightarrow \overline{M}_{0,0}(\mathbb{G}, \alpha)$, $H^1(\mathrm{ev}^* V|_C) = 0$ and $\pi_*\mathrm{ev}^* V$ is a vector bundle.  Global generation of $\pi_*\mathrm{ev}^* V$ follows from surjectivity of $H^0(\mathbb{G}, V) \rightarrow H^0(C, V|_C)$ for each fiber $C$.  Hence, for $X$ general in moduli,  $\overline{M}_{0,0}(X,\alpha)$ is smooth.
\end{proof}

\begin{proof}[Proof of Theorem \ref{thm: low degree curves}]

    By Lemmas \ref{lem: irred fiber} and \ref{lem: irred total space}, it suffices to show that if $H \cdot \alpha \leq 2$, then the general map parameterized by any component of $\overline{M}_{0,0}(X, \alpha)$ is either an embedding or, when $\dim X = 4$, a double cover of an $H$-line.  This is clear when $H\cdot \alpha = 1$, since $|H|$ is base point free.  We will address the case $H\cdot \alpha = 2$ below.

    By Lemma \ref{lem: fiber dimension}, stable maps with reducible domains cannot form a component of $\overline{M}_{0,0}(X, \alpha)$ when $H\cdot \alpha = 2$.  If $|H|$ is very ample, our claim follows immediately. %
    Otherwise, if $|H|$ is not very ample, the natural map $\pi \colon X \rightarrow \mathbb{P}|H|$ is a double cover branched over a smooth hypersurface $B$ of even degree, and the general curve may be the preimage of a line tangent to $B$ at $\frac{\deg B}{2} - 1$ points.  However, because $\dim X > 3$, the family of curves through a fixed point on the ramification divisor of $\pi$ has positive dimension.  If $C$ is a conic whose image in $\mathbb{P}|H|$ is a line tangent to $B$ at a point $q$, then deformations of $C$ through $q$ must remain tangent to $B$ at $q$ by the Riemann-Hurwitz formula. Hence, $C$ degenerates to a reducible curve by bend-and-break.  The locus of stable maps parameterizing reducible degenerations of $C$ has dimension $3\dim X - 8$.  However, Lemma \ref{lem: fiber dimension} and Lemma \ref{lem: cones of lines} show there can only be a $3\dim X - 9$ parameter family of stable maps whose unique node maps to the ramification divisor.  Hence, the stable map generalizes to nodal union of two $H$-lines with distinct images in $\mathbb{P}|H|$.
\end{proof}

\subsection{Rational curves on varieties with $g(X) > 11$}\label{subsection: contractible divisors}
Let $X$ be a coindex $3$ Fano variety of Picard rank $2$ and $g(X)>11$.
When $g(X) = 13$, $X$ is a homogeneous variety.
Hence Theorem \ref{thm: low degree curves} follows from \cite{Thomsen1998}, \cite{KimPandharipande2001}. 
Otherwise, $X$ has a contractible divisor.  There is a unique $\alpha \in \mathrm{Nef}_{1}(X)_{\mathbb{Z}}$ with $H\cdot \alpha = 1$; for this $\alpha$, $\overline{M}_{0,0}(X, \alpha)$ is irreducible, smooth, and generically parameterizes free curves.  This proves Theorem \ref{thm: low degree curves}(1).  
The following lemma describe components of $\overline{M}_{0,0}(X, \alpha)$ for each $\alpha \in \mathrm{Nef}_{1}(X)$ such that $\alpha \notin \partial\NE(X)$.

\begin{lem}\label{lem: mbb contractible divisor}
    Let $X$ be a smooth coindex $3$ Fano variety of dimension at least $4$.  Suppose $X$ is not a Fano $4$-fold of product type and $X$ contains a contractible divisor.  Then for all $\alpha \in \Nef_1(X)_{\mathbb{Z}}$ with $\alpha \notin \partial\NE(X)$, $\Mor(\mathbb{P}^1,X, \alpha)$ has a unique component which generically parameterizes free curves.
\end{lem}

\begin{proof}
    First, suppose $\dim X \geq 5$.  In this case, as $X$ contains a subvariety $Y$ with $a(Y,-K_X) > 1$, $X$ is the blow-up of $\mathbb{P}^5$ along a line $\ell$.  There are only two orbits of points under the action of $\mathrm{Aut}(X)$, the exceptional divisor $E$ and its complement.  For an irreducible curve $f \colon C \rightarrow X$ with $f(C) \not\subset E$, we may find a complimentary linear subspace $V \cong \mathbb{P}^3 \subset \mathbb{P}^5$ disjoint from $\ell$ and $f(C)$.  There is a $\mathbb{G}_m$-action on $X$ induced by the $\mathbb{G}_m$-action on $\mathbb{P}^5$ which fixes $\ell$ and $V$ pointwise.  Under an inclusion $\mathbb{G}_m \rightarrow \mathbb{P}^1$, the family $(g, g \circ f(C)) \subset \mathbb{G}_m \times X$ limits to a map $h \colon C \rightarrow X$ with $h(C) \subset E$.  Since $h^*\mathcal{O}_X(E)$ is a globally generated line bundle and $E$ is a homogeneous variety, %
    $h^{*}T_X$ is globally generated.  Since $E \cong \mathbb{P}^1 \times \mathbb{P}^3$, $\overline{M}_{0,0}(E, h_*[C])$ is irreducible, proving irreducibility of the locus in $\overline{M}_{0,0}(X, f_*[C])$ parameterizing irreducible curves. 

Suppose $\dim X = 4$ instead.  Since there is a divisor $Y$ with $a(Y,-K_X) > 1$, $\rho(X) = 2$ and $g(X) \in \{12, 14, 16, 21\}$.  Unless $g(X) = 14$, $Y$ contains every subvariety $Y' \subset X$ with $a(Y', -K_X) > 1$.  When $g(X) > 14$, the automorphism group of $X$ acts transitively on $X \setminus Y$ and on $Y$.  When $g(X) = 12$, $X$ is the blow-up of a smooth quadric $Q \subset \mathbb{P}^5$ along a conic whose linear span $V$ is not contained in $Q$.  In this case, the action of  $\mathrm{Aut}(X)$ on $X\setminus Y$ has two orbits: $Q \cap V'$ where $V'$ is the orthogonal complement to $V$, and $X \setminus Y \cup V'$.  Since $V' \cap X$ has codimension at least $2$, we may use this as before to limit a general free curve $f\colon C \rightarrow X$ to a curve $h \colon C \rightarrow Y$ with $h^{*}T_X$ globally generated.  As $Y \cong \mathbb{P}^1 \times \mathbb{P}^2$ in this case, irreducibility of the locus of free curves in $\overline{M}_{0,0}(X, f_*[C])$ follows.  When $g(X) > 14$, we may perform a similar limiting operation, but this time the limit of $f \colon C \rightarrow X$ will be a nodal curve $h \colon C' \rightarrow X$ where all but one component $C\subset C'$ are contracted under composition of $h$ with the elementary contraction of fiber-type on $X$.  For a general free curve $f \colon C  \rightarrow X$, the number of irreducible components of $C'$ will be $1 + E\cdot(f_*[C] - h_*[C])$.  This shows the general limit map $[h] \in\overline{M}_{0,0}(X, f_*[C])$ is a smooth point.  As the locus of such curves is irreducible, this also proves irreducibility of the locus of free curves in $\overline{M}_{0,0}(X, f_*[C])$.  

Lastly, suppose $g(X) = 14$, i.e., $X$ is the blow-up of a smooth quadric $Q \subset \mathbb{P}^5$ along a line $\ell$.  In this case, there is an elementary contraction $\pi \colon X \rightarrow \mathbb{P}^3$ of fiber-type.  Let $f \colon C \rightarrow X$ be a free curve general in its variety of deformations.  Since $\alpha \notin \partial\NE(X)$, a tangent bundle calculation shows $H^0(C, f^{*}T_X) \rightarrow H^0(C, (\pi \circ f)^{*}T_{\mathbb{P}^3})$ is surjective, so that the induced map $\pi_* \colon \overline{M}_{0,0}(X, \alpha) \rightarrow  \overline{M}_{0,0}(\mathbb{P}^3, \pi_*\alpha)$ is dominant at $[f]$.  However, $\overline{M}_{0,0}(\mathbb{P}^3, \pi_*\alpha)$ is irreducible, and the fiber of $\pi_*$ over a general point parameterizes sections of a Hirzeburch surface whose base is the domain of the map. %
The rigid section of this Hirzeburch surface lies on the unique contractible divisor $Y \subset X$, where $\pi|_Y \colon Y \rightarrow \mathbb{P}^3$ is the blow-up of a line.  Thus, the locus in $\overline{M}_{0,0}(X, \alpha)$ parameterizing free curves is irreducible.  
\end{proof}

\subsection{Spaces of quartic rational curves}
In Section \ref{section:a covers}, we will use the following lemma about general embedded quartic curves on a smooth coindex $3$ Fano $4$-fold $X$ whose fundamental linear series is not very ample. %
Recall that such varieties $X$ are smooth complete intersections in some weighted projective space.  This explicit description allows us to prove irreducibility of the space of embedded quartics on the universal complete intersection $\mathcal{X}$, mimicking Lemma \ref{lem: irred total space}.

\begin{lem}\label{lem: quartic very free}
    Let $X$ be a smooth Fano $4$-fold of coindex $3$ and Picard rank $1$ whose fundamental linear series $|H|$ is not very ample.  Suppose $X$ is general in moduli.  %
    If the general curve $C \subset X$ parameterized by a component of $\overline{M}_{0,0}(X)$ is a smooth $H$-quartic, then $C$ is very free.  %
    
\end{lem}

\begin{proof}
    Let $l$ be the class of an $H$-line.  Recall that $X$ is a general complete intersection of even degree hypersurfaces in some weighted projective space $\mathbb{G}$.  Let $\mathcal{X}$ be the universal complete intersection.  In these explicit circumstances, we claim the locus of complete intersections containing a given embedded quartic curve is of the expected dimension.  By the argument of Lemma \ref{lem: irred total space}, this proves a unique component of $\overline{M}_{0,0}(\mathcal{X},4l)$ generically parameterizes smooth embedded curves.

    To verify our claim, suppose $C \subset \mathbb{G}$ is an embedded rational quartic curve.  Consider the map $\phi : X \rightarrow \mathbb{P}^{g + 2}$ induced by the complete linear system $|H|$.  %
    Recall that $X$ is a branched double cover of its image $\overline{X} = \phi(X)$, and that the branch locus is defined by a hypersurface $B$ of degree at least $4$.  Moreover, the image $\overline{C} = \phi(C)$ of $C$ is a reduced curve of degree four everywhere tangent to $B$.
    
    First, we show that if $C\subset X$ corresponds to a general point of some component of $\overline{M}_{0,0}(X, [C])$, then $\overline{C}\subset  \mathbb{P}^{g + 2}$ is a rational normal quartic.  Castelnuovo-Mumford regularity for curves \cite{curvesRegularity} implies the natural map 
    $$H^0(\overline{X}, \mathcal{O}_{\overline{X}}(B)) \rightarrow H^0(\overline{C}, \mathcal{O}_{\overline{C}}(B))$$ 
    is surjective.  Since the arithmetic genus $p_a(\overline{C})$ is at most three, $\mathcal{O}_{\overline{C}}(B)$ is very ample, and being everywhere tangent to $B$ imposes the expected number of conditions.  Thus, the locus of double covers of $\overline{X}$ containing $C$ has dimension $c$ for some constant $c$ independent of $\overline{C}$.  However, the general quartic curve parameterized by $\Mor(\mathbb{P}^1, \overline{X})$ is a rational normal quartic.  Thus, as $X$ is general in moduli, $\overline{C}\subset  \mathbb{P}^{g + 2}$ is a rational normal quartic.

      As in Lemma \ref{lem: irred total space}, it follows that $\overline{M}_{0,0}(\mathcal{X},4l)$ has a unique component $N$ which generically parameterizes smooth embedded curves.
      There exists a boundary stratum $\partial N \subset N$ which parameterizes nodal unions of lines and smooth $H$-cubics.  A normal bundle calculation shows general smoothings of such curves are very free.  Indeed, consider a smooth fundamental divisor $Y$ on a general complete intersection $X$.  It is well known that $Y$ contains a family of free cubics and that the Hilbert scheme of lines on $Y$ is generically reduced.  If $C \subset Y$ denotes the general nodal union of a line and a cubic in $Y$, the exact sequence 
      $$0 \rightarrow N_{C/Y} \rightarrow N_{C/X} \rightarrow N_{Y/X}|_C$$
      shows $N_{C/X}$ is ample.  It follows that every irreducible component of the locus in $N$ that parameterizes curves in $X$ generically parameterizes very free curves.   %
\end{proof}

\section{Movable Bend-and-Break}\label{section:mbb}
In this section, we prove Movable bend-and-break for smooth Fano varieties of coindex $3$ and dimension at least $4$:
\theoremstyle{plain}
\newtheorem*{MainThmD}{\rm\bf Theorem \ref{MBB}}
\begin{MainThmD}
Let $X$ be a smooth coindex $3$ Fano variety of dimension $n \ge 4$.  Let $-K_X = (n-2)H$.  Suppose $X$ does not contain a contractible divisor.  Let $M \subset \overline{M}_{0,0}(X,\alpha)$ be a component that generically parameterizes free stable maps of $H$-degree at least $2$.  
Then $M$ contains a stable map $f\colon C_{1} \cup C_{2}\rightarrow X$ such that each restriction $f|_{C_{i}}\colon C_{i}\rightarrow X$ is a free morphism.
\end{MainThmD}

\begin{rem}
     Let $X$ be a coindex 3 Fano variety of dimension at least four that contains a contractible divisor.  Consider $\alpha \in \Nef_1(X)_{\mathbb{Z}}$.  If $\alpha \in \partial\NE(X)$, the conclusion of Theorem \ref{MBB} holds without modification.   If instead $\alpha \notin \partial \NE(X)$, Lemma \ref{lem: mbb contractible divisor} proves a unique component $M \subset \overline{M}_{0,0}(X,\alpha)$ generically parameterizes free curves.  In this case, $M$ contains a stable map as in Theorem \ref{MBB} if and only if $\alpha$ may be expressed as the sum of two integral nef classes.
\end{rem}

\subsection{Fibers of Families of Curves} We begin with a general lemma.

\begin{lem}\label{lem: fiber dimension}
    Let $X$ be a smooth Fano variety and let $Y \subset X$ be the union of all subvarieties $V \subset X$ with $a(V, -K_X |_V) > a(X, -K_X) = 1$.  Suppose $M \subset \Mor(\mathbb{P}^1, X)$ is a component whose universal family $\pi \colon \mathcal{U} \rightarrow M$ dominates $X$ under evaluation $\mathrm{ev}\colon \mathcal{U} \rightarrow X$.  For $m\geq 0$, consider the locus 
    $$Z_m =\{ p \in X \mid \dim \mathrm{ev}^{-1}(p) > \dim \mathcal{U} - \dim X + m\}.$$
    Then $\operatorname{codim}(Z_m \setminus Y, X) \geq 3 + m$.  
    
    Moreover, let $Z$ be a subvariety of $Z_m$.  Let $W$ be a component of $\mathrm{ev}^{-1}(Z)$ such that $W$ dominates $Z$ and $\dim W > \dim \mathcal{U} - \operatorname{codim}(Z, X) + m$.  Consider the closure $Z'$ of $\mathrm{ev}(\pi^{-1}(\pi(W)))$.  If $a(Z', -K_{X}) \le 1$, then $\operatorname{codim}(Z, Z') \geq 2 + m$.
\end{lem}

\begin{proof}
    The first claim follows from the second claim since rational curves parameterized by $\pi(W)$ are non-free, which are contained in a proper closed subset of $X$ by \cite[Theorem 3.11]{Kollar1996}.
    \par
    Let $Z$, $W$, and $Z'$ be as in the statement of Lemma \ref{lem: fiber dimension}.  
    Clearly, $Z \subset Z'$.
    Suppose that $\operatorname{codim}(Z, Z') \le 1 + m$.
    We claim that $a(Z', -K_{X}) > 1$.
    Take a smooth resolution $\phi\colon \tilde{Z}\rightarrow Z'$.
    Let $\tilde{W} \subset \Mor(\mathbb{P}^{1},\tilde{Z}, \alpha)$ be the locus consisting of strict transforms of general curves parametrized by $\pi(W)$.
    Note that $\tilde{W}$ parameterizes a dominant family of curves on $\tilde{Z}$.   
    Thus we have 
    \[-K_{\tilde{Z}} \cdot \alpha + \dim Z' \ge \dim \tilde{W} =  \dim \pi(W).\]
    \par
    First, suppose that $Z = Z'$.
    In this case, $\pi^{-1}(\pi(W)) = W$.  Hence,
    \begin{align*}
        \dim \pi(W) &= \dim W - 1\\
         &\ge \dim \mathcal{U} - \dim X + m + \dim Z\\
         &= -\phi^{*}K_{X}\cdot \alpha + 1 + m + \dim Z
    \end{align*}
    since $\mathrm{ev}\colon \mathcal{U} \rightarrow X$ is dominant.
    Thus, $(K_{\tilde{Z}} - \phi^{*}K_{X})\cdot \alpha \le - 1 - m < 0$.
    Since $\alpha$ is a nef curve class on $\tilde{Z}$, we conclude that $a(Z', -K_{X}) > 1$.
    \par
    Suppose that $Z \subsetneq Z'$.
    Then the restriction $\pi|_{W}$ is generically finite.  Hence,
    \begin{align*}
        \dim \pi(W) &= \dim W\\
         &> \dim \mathcal{U} - \dim X + m + \dim Z\\
         &= -\phi^{*}K_{X}\cdot \alpha + 1 + m + \dim Z.
    \end{align*}
    Thus, $(K_{\tilde{Z}} - \phi^{*}K_{X})\cdot \alpha < - 1 - m + \operatorname{codim}(Z,Z') \le 0$ and we again conclude that $a(Z', -K_{X}) > 1$.
\end{proof}

\begin{rem}\label{rem: cones of lines}
    Assume $X$ is a smooth coindex $3$ Fano variety.  In the notation of Lemma \ref{lem: fiber dimension}, if $\dim X = 4$ then  $Z_1\setminus Y$ is finite and $Z_2\setminus Y$ is empty.  If $\dim X \geq 5$, we will show the same statement holds in Lemma \ref{lem: cones of lines} and Corollary \ref{cor: Z1 finite}.  
\end{rem}

For many coindex 3 Fano varieties, Lemma \ref{lem: fiber dimension} may be used to prove Theorem \ref{MBB} from Theorems \ref{classification higher a} and \ref{thm: low degree curves}.  This is proven in Corollary \ref{cor: gluing cor}, which applies to all smooth, \textit{nearly level} Fano varieties as defined below.  This definition generalizes a concept of levelness found in \cite{RY2019}.

\begin{defn}[c.f.\ Definition~6.2, \cite{RY2019}]\label{def: nearly d-level}
    Let $X$ be a smooth $n$-dimensional Fano variety and $Y \subseteq X$ be a subvariety.
    For $d \in \mathbb{N}$, let $\overline{M}_{0,1}(Y,d)$ be the Kontsevich space of stable maps of $(-K_X)$-degree $d$.  Consider the evaluation map $\mathrm{ev}_{d}: \overline{M}_{0,1}(Y,d) \rightarrow Y$ and set
    $$Y_{d} = \{p \in Y | \dim \mathrm{ev}_d^{-1}(p) > d  - 2\}.$$ 
     We say that $Y$ is \textit{nearly $d$-level} if for all $e \leq d$ the following conditions are satisfied:
     \begin{enumerate}
        \item $\dim Y_{e} \leq n - 3$;
        \item for all $p \in Y_{e}$, $\dim \mathrm{ev}_{e}^{-1}(p) = e - 1$;
        \item for any component $W \subset \mathrm{ev}_{e}^{-1}(p)$ of dimension $e-1$, the image of a general map parameterized by $W$ meets $\bigcup_{f \leq \max (d, n-1)} Y_{f}$ at finitely many points.
     \end{enumerate}
    We call $Y$ \textit{nearly level} if $Y$ is nearly $d$-level for all $d \geq 1$.
\end{defn}

The motivation for Definition \ref{def: nearly d-level} is the following two results.  We use notation from Definition \ref{def: nearly d-level} to express and prove these results.

\begin{lem}\label{lem: dominant fiber product}
    Let $X$ be a nearly level smooth Fano variety.  For any two components $M_1', M_2' \subset \overline{M}_{0,1}(X)$, each irreducible component of the fiber product $M_1' \times_X M_2'$ dominates $X$.  
\end{lem}

\begin{proof}
    Let $W$ be an irreducible component of $M_1' \times_X M_2'$.   The fiber of $W$ over a point $p \in X$ either has at most the expected dimension or $p$ lies in a codimension $3$ subset and $\dim W_p$ is at most $2$ greater than expected.   Thus, as $W$ must have at least the expected dimension, $W \rightarrow X$ is dominant.
\end{proof}

\begin{prop}\label{prop: nearly level varieties}
    Let $X$ be a smooth $n$-dimensional Fano variety.  If $Y \subseteq X$ is nearly $(n-1)$-level, then $Y$ is nearly level.  Moreover, for all $d \in \mathbb{N}$, $Y_{d} \subseteq \bigcup_{e\leq n -1} Y_e$.
\end{prop}

We will prove Proposition \ref{prop: nearly level varieties} via the following two lemmas.

\begin{lem}\label{lem: nonempty boundary}
    Let $Y \subseteq X$ and $\mathrm{ev}_d: \overline{M}_{0,1}(Y,d) \to Y$ be as in Definition \ref{def: nearly d-level}.  Let $p \in Y$ and $W$ be an irreducible component of $\mathrm{ev}_d^{-1}(p)$.  If $d \geq \dim X$ and $\dim W > d - 2$, then $W$ contains a stable map with reducible domain.
\end{lem}

\begin{proof}
    Suppose to the contrary that $\dim W \geq d - 1$ and each stable map parameterized by $W$ has irreducible domain.  Then curves parameterized by $W$ are not free, so their images are contained in a subvariety $Z \subset X$ of dimension at most $\dim X - 1$ by \cite[Theorem 3.11]{Kollar1996}.  However, the universal family over $W$ has dimension at least $d \geq \dim X$.  Hence, $W$ contains a one-parameter family of curves passing through $p$ and another fixed point in $Z$.  By bend-and-break, $W$ must parameterize reducible curves as well, a contradiction.
\end{proof}

\begin{lem}\label{lem: reducible curves fibers}
    Let $X$ be a smooth $n$-dimensional Fano variety and $Y \subset X$ be a nearly $d$-level subvariety.  Consider the evaluation map $\mathrm{ev}_{d+1}: \overline{M}_{0,1}(Y,d+1) \to Y$.  Let $p \in Y$ and $W$ be a subvariety of $\mathrm{ev}_{d+1}^{-1}(p)$.  If $W$ generically parameterizes reducible curves, then $\dim W \leq d - 1$.  If in addition $\dim W = d - 1$, then $p$ lies in $\bigcup_{e \leq d} Y_e$ and the general map $C \rightarrow Y$ parameterized by $W$ does not lie in $\bigcup_{e\leq n - 1} Y_e$.
\end{lem}

\begin{proof}
    We proceed by induction on $d$.  The base case $d = 0$ is trivial.  Assume the inductive hypotheses; that is, if $d' \leq d$, $p' \in Y$, and $W' \subset \mathrm{ev}_{d'}^{-1}(p')$ parameterizes reducible curves, then $\dim W' \leq d' -2$.  If $\dim W' = d' -2$, then $p' \in \bigcup_{e \leq d'}Y_e$ and the image of a general map parameterized by $W'$ does not lie in $\bigcup_{e\leq n - 1} Y_e$.
    
    Let $f : (C, p_1) \to (Y,p)$ be a general map parameterized by $W$.  Let $q$ be a node of $C$ and let $C_1, C_2$ be the two connected components of $C\setminus \{q\}$.  We may assume the marked point $p_1 \in C$ lies in $C_1$.  Set $d_i : = -K_X \cdot C_i$ and note that $d_1, d_2 \leq d$.

    Consider the restrictions $f|_{C_i}$ as pointed stable maps $f_1 : (C_1, p_1, q_1) \to Y$ and $f_2 : (C_2, q_2) \to Y$.  Note that $W$ parameterizes pairs $(f_1', C_1', p_1', q_1'),(f_2', C_2', q_2')$ of deformations of $f_1, f_2$ such that $f_1'(p_1') = p$ and $f_1'(q_1') = f_2'(q_2')$.  Thus, there is a subvariety $M_1 \subset\mathrm{ev}_{d_1}^{-1}(p)$ containing $(f_1, C_1, p_1)$ such that forgetting $f_2$ and $q_1$ induces a surjective map $\psi \colon W \to M_1$.  The fiber $F = \psi^{-1}(f_1, C_1, p_1)$ parameterizes choices of the deformation $(f_2', C_2', q_2')$ and a point $q_1' \in C_1$ such that $f_2'(q_2') =f_1(q_1')$.  By generality of $f \in W$, it suffices to show $\dim M_1 + \dim F \leq d-1$, and if equality is achieved then $p \in \bigcup_{e\leq d} Y_e$ and $f(C)$ does not lie in $\bigcup_{e\leq n - 1} Y_e$.
    
    We divide our argument into three cases: $M_1$ has dimension at most $d_1 - 2$ and generically parameterizes irreducible curves; $M_1$ generically parameterizes reducible curves; or $M_1$ has dimension greater than $d_1 - 2$.  Before doing so, we describe a common use of the inductive hypothesis in the next paragraph. 
    
    Let $C_1' \subset C_1$ be the component containing $q_1$ (and therefore $q_1'$ for all maps in $F$).  By the inductive hypothesis, either deformations of $f_2$ meeting a general point of $f(C_1')$ have dimension at most $d_2 - 2$, or $C_2$ is irreducible and the fiber of $\operatorname{ev}_{d_2} \colon \overline{M}_{0,1}(Y, d_2) \to Y$ over every point in $f(C_1')$ has dimension $d_2 - 1$.  In the former case, $\dim F \leq d_2 - 1$ since $Y$ is nearly $d$-level.  In the latter case, $\dim F \leq d_2$ and $f(C_1')$ lies in the subvariety $Y_{d_2}$ where fibers of $\operatorname{ev}_{d_2}$ have dimension $d_2 - 1$.  

    We now proceed with our casework.  First, suppose $M_1$ generically parameterizes irreducible curves and $\dim M_1 \leq d_1 - 2$.  As $\dim F \leq d_2$, the dimension of $W$ is at most $d_1 + d_2 - 2 = d - 1$. If equality is achieved, then by the inductive hypothesis $C_2$ must be irreducible and $p$ must lie in $Y_{d_2}$.  Moreover, by generality of $f$, $(f_2, C_2, q_2)$ is a general map in $\operatorname{ev}_{d_2}^{-1}(f(C_1))$.  By Definition \ref{def: nearly d-level}(3), $f_2(C_2)$ does not lie in $\bigcup_{e\leq n-1} Y_{e}$, and so $f : C \to Y$ has the desired description.
    
    Next, suppose $M_1$ generically parameterizes reducible curves.  By the inductive hypothesis, $\dim M_1 \leq d_1 - 2$.  Hence, $\dim W \leq d - 1$ as before.  If equality is achieved, then $\dim F =d_2$ and $\dim M_1 = d_1 - 2$.  This implies that $f(C_1')$ lies in $Y_{d_2}$.  By generality of $f$ and the inductive hypothesis, $f(C_1)$ does not lie in $\bigcup_{e\leq n - 1} Y_e$.  This proves our claim.

    Lastly, suppose $\dim M_1 > d_1 - 2$.  By Definition \ref{def: nearly d-level}, $p \in Y_{d_1}$ and $\dim M_1 = d_1 -1$.  Thereore, by the inductive hypothesis $C_1$ is irreducible.  Note as well that $M_1$ is a component of $\operatorname{ev}_{d_1}^{-1}(p)$.  As $(f_1,C_1, p_1)$ is general in $M_1$, Definition \ref{def: nearly d-level}(3) implies $f(C_1)$ meets $Y_{d_2}$ and $\bigcup_{e\leq n-1} Y_{e}$ at finitely many points.  Hence, $\dim F \leq d_2 - 1$ and $f(C)$ does not lie in $\bigcup_{e\leq n-1} Y_{e}$.  This finishes our proof.
\end{proof}

We are now ready to prove Proposition \ref{prop: nearly level varieties}.

\begin{proof}[Proof of Proposition \ref{prop: nearly level varieties}]
Let $X$ be a smooth $n$-dimensional Fano variety.  Suppose $Y \subseteq X$ is nearly $(n-1)$-level subvariety.  We will prove by induction on $d$ that $Y$ is nearly $d$-level and $Y_{d'} \subseteq \bigcup_{e\leq n-1} Y_{e}$ for all $d' \leq d$.  To do so, we may suppose this statement is true for some $d \geq n -1$ and prove it is true for $d + 1$.

To demonstrate this, suppose $p \in Y_{d+1}$ and $W \subset \mathrm{ev}_{d+1}^{-1}(p)$ is an irreducible component such that $\dim W > d - 1$.  As $Y$ is nearly $d$-level, by Lemma \ref{lem: reducible curves fibers} the general curve parameterized by $W$ is irreducible.  However, because $d + 1 \geq n$, $W$ contains a stable map with reducible domain by Lemma \ref{lem: nonempty boundary}.  Let $\partial W \subset W$ be a component of the locus of stable maps with reducible domain.  

Note that $\operatorname{codim}(\partial W , W) = 1$.  By Lemma \ref{lem: reducible curves fibers} $\dim \partial W \leq d - 1$.  Hence, as $\dim W \geq d$, equality is achieved and $\dim \partial W = d - 1$.  Lemma \ref{lem: reducible curves fibers} therefore shows $p$ must lie in $\bigcup_{e\leq d} Y_{e}$ and the general map parameterized by $\partial W$ is not contained in $\bigcup_{e\leq n-1} Y_{e}$.  However, $\bigcup_{e\leq d} Y_{e} = \bigcup_{e\leq n-1} Y_{e}$ by induction.  As the general map $f : C \to Y$ parameterized by $W$ has irreducible domain, $f(C)$ meets $\bigcup_{e\leq d} Y_{e}$ at finitely many points.  This proves $Y$ is nearly $(d+1)$-level and $Y_{d+1} \subseteq \bigcup_{e\leq n-1} Y_{e}$.  
\end{proof}

\begin{rem} 
    General coindex 3 Fano varieties of dimension $n \geq 6$ are nearly level.  Indeed, let $X$ be such a variety and let $H$ be a fundamental divisor on $X$. Note that $\rho(X) = 1$ by Theorem  \ref{classification Picard rank 2}.  Theorem \ref{thm: low degree curves}(1) shows the family of $H$-lines on $X$ is smooth.  By an analysis of normal bundles, this implies there is no point $p \in X$ such that a family of $H$-lines through $p$ has dimension $n - 2$.  Moreover, Theorem \ref{classification higher a} proves there are no subvarieties $Y \subset X$ such that $a(Y,-K_X) > 1$.  Hence, Lemma \ref{lem: fiber dimension} implies $X$ is nearly $(n-2)$-level.  Since $2(n-2) > n - 1$, $X$ is nearly level by Proposition \ref{prop: nearly level varieties}.  In particular, the following corollary applies to $X$.
\end{rem}

\begin{cor}\label{cor: gluing cor}
    Let $X$ be a smooth nearly level Fano variety.  For some $r > 0$, suppose any component of $\overline{M}_{0,0}(X,d)$ parameterizing rational curves of anticanonical degree $r < d \leq \dim X + 1$ has nonempty boundary.  Then every component of $\overline{M}_{0,0}(X)$ generically parameterizes free curves and contains a stable map whose irreducible components are free curves of $-K_X$-degree at most $r$.
\end{cor}

\begin{proof}
    For any $\alpha \in \NE(X)$ and any component $M \subset \overline{M}_{0,1}(X, \alpha)$, Definition \ref{def: nearly d-level} implies the evaluation map $\mathrm{ev} : M \to X$ is dominant.  Thus, $M$ generically parameterizes free curves.  Moreover, any boundary component of $M$ contains a stable map with free components by Lemma \ref{lem: dominant fiber product}.  We may find the desired stable map by deforming to a higher codimension boundary strata.
\end{proof}

This proves Theorem \ref{MBB} for general coindex 3 Fano varieties of dimension at least 6.  We expect that for general Fano varieties of many other deformation types, a Movable bend-and-break result will follow from similar variations of levelness.  This sort of strategy does not work for arbitrary Fano varieties in a family, as levelness tends not to specialize well. Instead, to extend our result to arbitrary coindex $3$ Fano varieties with no contractible divisors, we will require the following lemma and its corollary.  These analyze and constrain the failure of Definition \ref{def: nearly d-level}(1-2).

\begin{lem}\label{lem: cones of lines}
    Let $X$ be a smooth Fano variety of coindex $3$ and dimension at least $4$.  Let $\pi \colon U \rightarrow M$ be the universal family over a component $M \subset \Mor(\mathbb{P}^1, X)$ parameterizing lines.
    Suppose a component $U_p$ of the fiber $\mathrm{ev}^{-1}(p)$ over a point $p \in X$ has dimension at least $\dim U - \dim X + 2 = \dim X + 1$.  Then $\dim U_p = \dim X + 1$, $\rho(X) = 1$, and $X$ is either a quartic hypersurface or $g(X) = 2$. %
    Furthermore:
    \begin{enumerate}
        \item There are finitely many such points $p \in X$.
        \item $M_p = \pi(U_p)$ equipped with the reduced structure is smooth. %
    \end{enumerate}
\end{lem}
\begin{proof}
    Let $D = \mathrm{ev}(\pi^{-1}(M_p))$.  Then as the natural map $f \colon X \rightarrow \mathbb{P}|H|$ is finite onto its image, only finitely many lines parameterized by $M_p$ may pass through a point $q \in X \setminus \{p\}$.  Thus, $\dim D = \dim U_p + 1 - \dim \mathrm{Aut}(\mathbb{P}^{1}) \geq \dim X - 1$.  Since lines parameterized by $M_p$ are not free, they cannot pass through a general point of $X$.  This shows $\dim U_p = \dim X + 1$. 
    
    Suppose there is a one parameter family $p_t$ of points in $X$ such that $\dim U_{p_t} = \dim X + 1$.  Let $D_t$ be the divisor covered by non-free lines through $p_t$.  Then algebraic family $\{D_t\}$ would dominate $X$, contradicting the  containment of non-free curves of bounded degree in a proper closed subset (\cite[Theorem 3.11]{Kollar1996}).  Hence, $p \in X$ is one of finitely many points such that $\dim U_p = \dim X + 1$.

    Let $Y \subset X$ be a smooth fundamental divisor containing $p$.  Since $Y$ is the preimage of a hyperplane under $f : X \to \mathbb{P}|H|$, the intersection $D \cap Y \subset Y$ is a divisor swept out by $H$-lines through $p$ as well.  By descending induction on the dimension of $X$, we may assume $\dim X = 3$.  It follows from \cite[Proposition~6.4]{BurkeJovinelly2022} that $\rho(X) = 1$ and $g(X) \leq 3$.  We claim that if $g(X) = 3$, the fundamental linear system must be very ample.  This will follow from our proof of smoothness of $M_p$.

    First consider the case when $X$ is a quartic hypersurface.  In this case, $D \subset X$ must be the tangent hyperplane section of $X$ at $p$.  Since the Gauss map is finite on a smooth hypersurface, $D$ must be smooth away from $p$.  As we may identify $M_p$ with a general hyperplane section of $D$, it follows that $M_p$ is smooth.

    Suppose instead that $H$ is not very ample.  Let $f \colon X \rightarrow \mathbb{P}|H|$ be the natural map, and $R \subset X$ the ramification divisor of $f$.  Once more, $f(D)$ must be contained in the linear span of $df_p \colon T_p X \rightarrow T_{f(p)} \mathbb{P}^n$.  This has two consequences.
    
    First, it shows $g(X) = 2$.  Indeed, suppose $g(X) = 3$ instead.  Note that $f(X) \subset \mathbb{P}|H|$ is a smooth quadric $Q$.  If $p \in R$, then $df_p$ would have nontrivial kernel.  This would imply that $f(D)$ lies in a codimension $2$ linear section of $Q$, which is impossible since $\dim D = \dim X - 1$ and each hyperplane section of $Q$ is irreducible.  If instead $p \notin R$, the fiber of $f : X \to Q$ over $p$ contains another point $p' \neq p$.  By the preceding paragraph, $f(D)$ must be the tangent hyperplane to $Q$ at $f(p)$.  In particular, $f^{-1}(f(D))$ is irreducible and therefore equal to $D$.  However, none of the lines passing through $p$ meet $p'$.  Hence, $D \subsetneq f^{-1}(f(D))$, a contradiction.

    Next, when $g(X) = 2$, it shows that if $df_p$ is not injective, i.e., if $p \in R$, then $f(D)$ must be the preimage of the tangent hyperplane $T$ to the branch divisor $B = f(R)$ at $f(p)$.  In this case, by finiteness of the Gauss map for $B$, we see that $T \cap B$ has finitely many singular points.  We claim that $T \cap B$ must be the cone over a smooth sextic hypersurface of dimension $\dim X - 2$.  Indeed, each line in $T$ through $f(p)$ must be tritangent to $B$.  This may only happen if $f(p)$ is a point of multiplicity $6$ in $T \cap B$, proving that $T \cap B$ is a cone.  Smoothness of the base of $T \cap B$ follows from finiteness of the Gauss map once more.  Hence, as the base $M_p$ of the family of lines through $p$ is a double cover of a general hyperplane in $T$ branched over its intersection with the sextic cone $T \cap B$, (2) holds when $p \in R$.  %
    
    Lastly, when $p \notin R$, each line parameterized by $U_p$ has normal bundle $N \cong \mathcal{O}(-2) \oplus \mathcal{O}(1)^{\oplus \dim X - 2}$.  %
    Thus, the reduced structure on $\pi(U_p)$ is clearly smooth.
\end{proof}

\begin{exmp}\label{example: big fiber dim}
    There are indeed double covers $f\colon X\rightarrow \mathbb{P}^{n}$ branched along a smooth sextics $B \subset \mathbb{P}^{n}$ having points $p$ in the ramification locus such that lines through $p$ form an $(n-2)$-dimensional family.
    \par
    Let $B = V(x_{0}^{6} + \dots + x_{n}^{6})$ be the Fermat sextic and let $p$ be one of two points lying over $q = (1:0:\dots:0)\in \mathbb{P}^{n}$.
    Note that the space of lines on $X$ passing through $p$ is isomorphic to the space of lines passing through $q$ and tritangent to $B$ via $f$.
    In this setting, the latter space is $V(x_{1}^{6} + \dots + x_{n}^{6}) \subset \mathbb{P}^{n-1}$, which is smooth.
    Hence the divisor swept out by lines through $p$ is isomorphic to the cone $V(x_{1}^{6} + \dots + x_{n}^{6}) \subset \mathbb{P}^{n}$.
    By symmetry, there are at least $2(n+1)$ such conical points.
    Moreover, one may be able to show that any other point cannot be conical.
\end{exmp}

\begin{cor}\label{cor: Z1 finite}
    Let $X$ be a smooth Fano variety of coindex $3$ and dimension at least $4$.  Let $\pi \colon U \rightarrow M$ be the universal family over a component $M \subset \Mor(\mathbb{P}^1, X)$.
    Suppose $\mathrm{ev} \colon U \rightarrow X$ is dominant.  For $m \in \{1,2\}$, consider the locus
     $$Z_m =\{ p \in X \mid \dim \mathrm{ev}^{-1}(p) > \dim U - \dim X + m\}.$$
    Then $Z_1$ is finite and $Z_2$ is empty.  
    If $\dim X \geq 5$ and $Z_{1}$ is nonempty, then $M$ parameterizes lines.
\end{cor}

\begin{proof}
    Let $Z$ be a component of $Z_m$.  There exists a component $W \subset \mathrm{ev}^{-1}(Z)$ such that $W$ dominates $Z$ and $\dim W > \dim U - \operatorname{codim}(Z,X) + m$.  Let $Z'$ be the closure of $\mathrm{ev}(\pi^{-1}(\pi(W)))$.  Since curves parameterized by $\pi(W)$ are not free, $Z'$ is a proper subvariety of $X$ and $\pi(W)$ is a proper subvariety of $M$.

    First suppose $\dim X = 4$.  Since $\dim M > \dim \pi(W)$, the lower bound on $\dim W$ shows $\dim Z \leq 1$.  The restriction of $\pi$ to $W$ must be generically finite, as $\dim Z' \geq 2$ by Lemma \ref{lem: fiber dimension} and \cite[Proposition 1.3]{Hoering2010}.  Thus, 
    \begin{align*}
        \dim M &> \dim W\\
        &\geq \dim U + \dim Z - 4 + m + 1\\
        & = \dim M + \dim Z - 2 + m.        
    \end{align*}
This proves our claim for the case $\dim X = 4$.

    Suppose $\dim X \geq 5$ instead.  If $X$ contains a subvariety $Y \subset X$ such that $a(Y,-K_X) > 1$, then $X$ is the blow-up of $\mathbb{P}^5$ along a line.  In this case, our claim also follows from Lemma \ref{lem: fiber dimension} since the action of $\mathrm{Aut}(X)$ on $X$ has no orbits of dimension less than $4$. 

    Suppose $X$ is not the blow-up of $\mathbb{P}^5$ along a line.  Let $q \in Z$ and $W_q \subseteq \mathrm{ev}^{-1}(q)$ be a component with $\dim W_q > \dim U - \dim X + m$.  Let $D$ be the closure of $\mathrm{ev}(\pi^{-1}(\pi(W_q)))$.  We will show $m = 1$, $Z = \{q\}$, and $W_q$ parameterizes $H$-lines.
    
    Because there are no subvarieties $Y \subset X$ with $a(Y,-K_X) > 1$, every component of $\Mor(\mathbb{P}^1, X)$ parameterizes a dominant family of curves.  Note, however, that $D$ is not a nearly level subvariety of $X$.  In particular, by Proposition \ref{prop: nearly level varieties} there must be a family $\mathcal{L}$ of $H$-lines in $D$ such that $\dim \mathcal{L} = \dim X - 2$ and every line in $\mathcal{L}$ passes through a common point $p \in D$.  By Lemma \ref{lem: cones of lines}, there are finitely many choices for $p \in X$.  To prove $Z = \{q\}$, it therefore suffices to show $q$ is one such choice.

    Let $d > 1$ and suppose $N \subset \Mor_d(\mathbb{P}^1, D, 0 \mapsto q)$ is a component of the morphism scheme of maps of $H$-degree $d$ that send $0$ to $q$.  We will show the dimension of $N$ is at most $d(\dim X - 2) + 1$.  Note that if $\dim N > d(\dim X - 2) + 1$, the subvariety of $X$ swept out by curves in $N$ is not nearly level, and so this subvariety must be $D$.  Hence, we may suppose $N$ parameterizes a dominant family of curves in $D$.  Working modulo automorphism of the maps, we show the image of $N$ in the Kontsevich space $\overline{M}_{0,0}(D)$ has dimension at most $d(\dim X - 2) - 1$.

    First, suppose $p = q$.  Recall that a general hyperplane section $H\cap D$ of $D$ is smooth by Lemma \ref{lem: cones of lines}(2).  Hence, $a(H \cap D, H) \leq \dim X - 4$.  Moreover, for some $r > 0$ there is a natural map $\pi \colon N \dashrightarrow \Mor_r(\mathbb{P}^1, H \cap D)$, induced by projection from $p$ onto $H \cap D$.  The fibers of $\pi$ are the linear series of the space of sections on the Hirzeburch surface $\mathbb{F}_r = \mathbb{P}_{\mathbb{P}^1}(\mathcal{O} \oplus \mathcal{O}(-r))$ which meet the rigid section at $d-r > 0$ points.  This linear series has dimension $2d - r + 1$.  As any dominant family of curves of $H$-degree $r$ in $H \cap D$ deforms with dimension at most $r(\dim X-4) + \dim H \cap D -3$, the total dimension of the family is at most $(r+ 1)(\dim X - 5) + 2d + 1$.  Since each curve passes through $p$, $r < d$, so that a maximum dimension of $d(\dim X - 3) + 1$ is achieved when $r = d - 1$.  As $d > 1$, this is strictly less than the advertised bound $d(\dim X - 2) - 1$.  Therefore, if $p = q$, $W_q$ parameterizes $H$-lines.

    When $q \neq p$, we still have a map $\pi\colon N \dashrightarrow \Mor_r(\mathbb{P}^1, H \cap D)$ induced by projection from $p$, but now possibly $r = d$ and the fibers of $\pi$ have dimension $1$ less, as each section must pass through $q$.  Moreover, the parameter space of curves through $\pi(q)$ has dimension at most $r(\dim X-4) -2$, as the fiber over $\pi(q)$ has the expected dimension since the general curve it parameterizes passes through a general point of $D \cap H$.  Hence, the total dimension of the family of curves through $q$ is at most 
    $$(r+ 1)(\dim X - 5) + 2d + 1 - \dim X + 2 \leq d(\dim X - 3) - 2$$
    which is again less than the advertised bound $d(\dim X - 2) - 1$. 
\end{proof}

As a corollary, we obtain Theorem \ref{thm: component reducible stable maps}.  

\begin{proof}[Proof of Theorem \ref{thm: component reducible stable maps}]
Let $X$ be a smooth coindex 3 Fano variety of dimension $n \geq 5$.  Suppose $M \subset \overline{M}_{0,0}(X,\alpha)$ generically parameterizes reducible curves.  Let $f : C \to X$ be a general map parameterized by $M$.  

First, suppose $X$ is the blow-up of $\mathbb{P}^5$ along a line.  The action of $\mathrm{Aut}(X)$ on $X$ has no orbits of dimension less than $4$.  Hence, Lemma \ref{lem: fiber dimension} shows the only components of $\Mor(\mathbb{P}^1,X)$ which are not flat over $X$ parameterize curves contained in the exceptional divisor.  This implies our claim.

Suppose instead that $X$ is not the blow-up of $\mathbb{P}^5$ along a line. If $X$ is nearly level, there cannot exist such a component $M$ by Lemma \ref{lem: reducible curves fibers}.  Thus, $\rho(X) = 1$, $g(X) \leq 3$, and $X$ is not general in moduli by Theorem \ref{thm: low degree curves}, Lemma \ref{lem: cones of lines}, and Corollary \ref{cor: Z1 finite}.  The indicated description of $C$ follows from a dimension count that is similar to the proof of Lemma \ref{lem: mbb precurser} but also incorporates moduli of contracted components.
\end{proof}

\subsection{Proof of Movable Bend-and-Break}
Before proving Theorem \ref{MBB}, we record the following observation. %

\begin{lem}\label{lem: mbb precurser}
    Let $X$ be a smooth coindex $3$ Fano variety of dimension at least $4$.  Let $T \subset \overline{M}_{0,0}(X, \alpha)$ be a constructible subset such that for every irreducible component $C_i$ of the domain of a map $f \colon C \rightarrow X$ parameterized by $T$, either $f$ contracts $C_i$ or $[f|_{C_i}] \in \overline{M}_{0,0}(X, f_*[C_i])$ lies in a component parameterizing a dominant family of curves.  Let $R_1, \dots, R_n$ be general complete intersections of basepoint free linear systems with $\operatorname{codim}(R_i, X) = r_i$.  Consider the sublocus $T_0 \subset T$ parameterizing curves which pass through $m \ge 1$ general points $p_1, \dots, p_m$ and each $R_i$.
    \begin{itemize}
        \item If $-K_X \cdot \alpha + \dim X - 3 = m(\dim X - 1) + \sum_{i = 1}^{n} (r_i-1)$, %
        then $\dim T_0 \leq 0$ and each curve parameterized by $T_0$ is irreducible and general in moduli.
        \item If $-K_X \cdot \alpha + \dim X - 4 = m(\dim X - 1) + \sum_{i = 1}^{n} (r_i-1)$, then $\dim T_0 \leq 1$ and at most finitely many points in $T_0$ parameterize reducible curves, each of which must be the nodal union of two free curves.   %
    \end{itemize}
\end{lem}
\begin{proof}
    Let $f \colon C \rightarrow X$ be any curve parameterized by $T_0$.  If $C$ is irreducible, there is nothing to prove; hence, we may suppose $C = \cup C_i$ is a stable nodal curve with at least two irreducible components.  We will show the image of the map $f \colon C \rightarrow X$ deforms with dimension at most $-K_X \cdot \alpha + \dim X - 4$ while preserving the dual graph to $C$. Moreover, if this family of deformations has dimension $-K_X \cdot \alpha + \dim X - 4$, we show $C$ contains two irreducible components and the restriction of $f$ to each component is a free curve.  Since meeting a general codimension $r_i$ subvariety imposes $r_i - 1$ conditions on the family of images of deformations of $f \colon C \rightarrow X$, our claim follows immediately from these deductions.

    Let $C_1$ be an irreducible component of $C$ such that $f(C_1)$ contains $p_1$.  By generality of $p_1$, $f|_{C_1}$ is a free curve.  Let $M_1 \subset \overline{M}_{0,1}(X, f_*[C_1])$ be the sublocus parameterizing free curves.  The evaluation map $\mathrm{ev}_1\colon M_1 \rightarrow X$ is flat; in particular, any nonempty fiber of $\mathrm{ev}_i$ has the expected dimension.
    
    Let $C_2 \subset C$ be a component attached to $C_1$.  Let $M_2 \subset M_{0,1}(X, f_*[C_2])$ be a component containing $[f|_{C_2}]$ and let $\mathrm{ev}_2 : M_2 \to X$ be the evaluation map.  There is a natural gluing map 
    $$\mathrm{gl} \colon M_1 \times_X M_2 \rightarrow \overline{M}_{0,0}(X, f_*[C_1 \cup C_2])$$ 
    whose domain is the fiber product of $\mathrm{ev}_1$ and $\mathrm{ev}_2$.  Deformations of $f|_{C_1 \cup C_2}$ which preserve the dual graph of $C_1 \cup C_2$ are parameterized by the image of $\mathrm{gl}$.  Let $h : C_1 \cup C_2 \to X$ be a general map parameterized by $\mathrm{gl}(M_1 \times_X M_2)$.  

    Since $\mathrm{ev}_2 \colon M_2 \rightarrow X$ is dominant, a dimension count shows the fiber of $\mathrm{ev}_2$ is of dimension $n$ larger than expected, i.e.\ $-K_X \cdot f_*[C_2] - 2 + n$, only over a codimension $n + 1$ subset.  As $\mathrm{ev}_1$ is flat, this implies the dimension of $M_1 \times_X M_2$ is at most $-K_X \cdot f_*[C_1 \cup C_2] + \dim X - 4$.  When equality is achieved, the map $h$ sends the node of $C_1 \cup C_2$ to a general point of $h(C_1)$.  In this case, $h$ is a general union of two free curves.

    Suppose $\dim M_1 \times_X M_2 = -K_X \cdot f_*[C_1 \cup C_2] + \dim X - 4$.  If $C = C_1 \cup C_2$, then $f$ must be a general map parameterized by $\mathrm{gl}(M_1 \times_X M_2)$, proving our claim.  Otherwise, for each irreducible component $C_i \subset C$ that is not contracted by $f$, let $M_i \subset M_{0,1}(X, f_*[C_i])$ be the component parameterizing deformations of $[f|_{C_i}]$.  By Corollary \ref{cor: Z1 finite}, the fiber of $\mathrm{ev}_i \colon M_i \rightarrow X$ has dimension $-K_X \cdot f_*[C_i]$ over at most finitely many points.  Moreover, the image of $h$ is disjoint from the locus where the fiber dimension of $\mathrm{ev}_i$ jumps.  As $C$ is connected, the image of $f$ deforms with dimension less than $-K_X \cdot \alpha + \dim X - 4$ while preserving the dual graph of $C$.
\end{proof}

We are now ready to prove Theorem \ref{MBB}.  We relegate a few cases of the proof to the following lemma.

\begin{lem}\label{lem: MBB cases}
    Let $X$ and $M \subset \overline{M}_{0,0}(X,\alpha)$ be as in Theorem \ref{MBB}.  If $M$ does not generically parameterize immersions or $H \cdot \alpha = 2$, then $M$ contains a stable map $f : C_1 \cup C_2 \to X$ such that each restriction $f|_{C_i} : C_i \to X$ is a free morphism.
\end{lem}

\begin{proof}
 Suppose the general curve parameterized by $M$ is not immersed.  A dimension count shows $\dim X = 4$ and $M$ parameterizes multiple covers of $H$-lines.  We may degenerate the domain of such a cover to a reducible stable curve without changing the image of the map to $X$.  Thus, we may assume $M$ generically parameterizes immersions.

 Suppose instead that $H \cdot \alpha = 2$.  If $\alpha \in \partial\NE(X)$, let $\pi : X \to B$ be the contraction associated to the minimal face of $\NE(X)$ containing $\alpha$.  A general fiber $F \subset X$ of $\pi$ is a smooth Fano variety of dimension at most $n-2$ and index $n-2$.  Thus, the linear system $|H|$ embeds $F$ as either a linear subspace or a smooth quadric.  In both cases we may degenerate curves of class $\alpha$ on $F$ to a reducible union of free curves.

 It remains to consider the case $H\cdot \alpha = 2$ and $\alpha \notin \partial\NE(X)$.  If $X$ is general in moduli, Theorem \ref{thm: low degree curves} proves $M$ is the only component of $\overline{M}_{0,0}(X, \alpha)$ that parameterizes immersions.  Theorem \ref{classification a-covers} and Lemma \ref{lem: specialize Manin components} (which are independent of Theorem \ref{MBB}) show this statement is also true for arbitrary $X$.  Thus, it suffices to find a reducible curve of class $\alpha$ whose components are free.  
 
 If $\rho(X) = 2$, then both extremal contractions of $X$ are of fiber type since $X$ does not contain contractible divisors.  Hence, $\alpha$ is the sum of two distinct integral nef classes, both of which are represented by free curves.  
 
 If $\rho(X) = 1$, $\alpha = 2\ell$ for the unique $H$-line class $\ell$.  For any component $U \subset \overline{M}_{0,1}(X, \ell)$, we claim the evaluation map $\mathrm{ev} : U \to X$ cannot be birational.  This claim is trivial if $\dim X > 4$, as $\mathrm{ev}$ would have positive dimension fibers.  If instead $\dim X = 4$, let $s : \mathbb{P}^1 \to X$ be a general map parameterized by $M$.  Since $s$ is free, $s^*T_X \cong \mathcal{O}_{\mathbb{P}^1}(2) \oplus \mathcal{O}_{\mathbb{P}^1}^{\oplus 3}$.  If $\mathrm{ev}$ were birational, then by \cite[Proposition~3.1]{patel2020moduli} $T_X$ would be unstable.  This is a contradiction, as it is well known that $T_X$ is stable when $X$ is a coindex 3 Fano variety of Picard rank 1.  Hence, there are at least two distinct $H$-lines passing through a general point of $X$.  Gluing these together produces the desired map.
\end{proof}

\begin{proof}[Proof of Theorem \ref{MBB}]
Let $n = \dim X$.  Suppose $M \subset \overline{M}_{0,0}(X, \alpha)$ is a component that generically parameterizes irreducible free curves and $H \cdot \alpha > 1$.  By Lemma \ref{lem: MBB cases} we may suppose that $H\cdot \alpha > 2$ and that the normal bundle $N$ of the general curve parameterized by $M$ splits as a direct sum $N \cong \oplus_{i = 0}^{n-2} \mathcal{O}(a_i)$ with $a_i \leq a_{i + 1}$.

Let  $p_0, \ldots, p_{a_0} \in X$ be a general set of $a_0 + 1$ points.  Recall that $3H$ is very ample on $X$.  For $i \geq 1$, let $R_{i,1}, \ldots , R_{i, a_{i} - a_{i - 1}} \subset X$ be general dimension $i$ complete intersections of members of $|3H|$.  If $a_0 < a_{n-2}$, let $k \leq n-2$ be the maximal integer such that $a_k > a_{k-1}$.  Set $R = R_{k, a_k - a_{k-1}}$, and let $A_1, \ldots , A_{m} \in |3H|$ be the divisors cutting out $R$.  If instead $a_0 = a_{n-2}$ we let $R = p_{a_0}$ and consider $R$ as a local complete intersection of divisors $A_1, \ldots , A_m$.  Let $R' = A_1 \cap \ldots \cap A_{m-1}$.

Let $\widetilde{T} \subset M$ be the subfamily parameterizing all maps which meet each $p_j$ and each $R_{i,j}$ aside from $R$.  By construction, each curve parameterized by $\widetilde{T}$ meets every $R_{i,j}$ (aside from $R$) along a fixed point $q_{i,j} \in R_{i,j}$.  By varying the points $p_i$, we may assume $q_{i,j}$ is a general point of $R_{i,j}$, and therefore of $X$ as well.  Consider the locus $T \subset \widetilde{T}$ of curves which also meet $R'$.  By generality of our construction, $\dim T = 1$.  Since $H \cdot \alpha \geq 3$, $c_1(N) \geq 3(n-2) - 2 \geq n$.  In particular, either $a_{0} \geq 2$ or there are at least $2-a_0$ values of $i \leq n -2$ such that $a_{i} > a_{i-1}$. In each case, $\widetilde{T}$ parameterizes curves passing through at least two fixed points.  By bend-and-break, $T$ therefore parameterizes a stable map $f \colon C \to X$ with reducible domain.  

Let $C' \subset C$ be the union of all irreducible components $C_i \subset C$ such that $f$ contracts $C_i$ or $[f|_{C_i}]$ lies in a component of $\overline{M}_{0,0}(X, f_*[C_i])$ that parameterizes a dominant family of curves.  First, suppose $C' = C$.  Note that $\operatorname{codim}(R_{i,j}, X) = n - i$ and $\operatorname{codim}(R', X) =  \operatorname{codim}(R, X) - 1$.  By Lemma \ref{lem: mbb precurser}, as
$$-K_X \cdot \alpha + n - 4 = (a_0+ 1)(n - 1) + \sum_{i = 1}^{n-2} (a_{i} - a_{i-1})(n - 1 - i) - 1,$$
$C = C_1 \cup C_2$ has two irreducible components and $f|_{C_i}$ is a free curve for $i = 1,2$.  

Suppose instead that $C' \neq C$.  Let $C_i \subset C$ be an irreducible component such that $[f|_{C_i}] \in \overline{M}_{0,0}(X, f_*[C_i])$ only lies in components parameterizing non-dominant families of curves.  It follows that deformations of $[f|_{C_i}]$ are contained in a subvariety $Y_i \subset X$ with $a(Y_i, -K_X) > 1$.  By hypothesis, $X$ does not contain any contractible divisors.  Theorem \ref{classification higher a} therefore implies $\dim X = 4$ and $\operatorname{codim}(Y_i, X) = 2$ for each subvariety $Y_i \subset X$ with $a(Y_i, -K_X) > 1$.

Let $C_1', \ldots , C_t' \subset C'$ be the connected components of $C'$.  The intersection of each $R_{i,j} \neq R$ with any curve parameterized by $T$ contains the point $q_{i,j} \in R_{i,j}$.  As $q_{i,j}$ is a general point of $X$, it does not lie in a subvariety $Y \subset X$ with $a(Y, -K_X) > 1$.  It follows that for each $R_{i,j} \neq R$, the image $f(\cup_s C_s')$ meets $R_{i,j}$ along $q_{i,j}$.

If $\operatorname{codim}(R', X) \neq 2$, we claim $f(\cup_s C_s')$ meets $R'$ as well.  Indeed, the image of the family of curves over $\widetilde{T}$ is a subvariety $S \subseteq X$ with $\dim S = \operatorname{codim}(R, X) \geq 2$.  If $\dim S = 2$, $R'$ is a divisor and $T = T'$.  If $\dim S = 4$, then $S = X$ and $R'$ is a general complete intersection curve in $X$.  In this case, $R'$ is disjoint from any of the codimension two subvarieties $Y_i \subset X$ with $a(Y_i, -K_X) > 1$.

We will use the incidence of $f(\cup_s C_s')$ with $R_{i,j} \neq R$ and, when it applies, with $R'$ to derive a contradiction.  Recall that we may assume $\dim X = 4$ and $X$ does not contain a contractible divisor.  For $s \leq t$, let $a_{s,0}$ be the number of general points $p_j$ met by $f(C_s')$.  Similarly, let $a_{s,i}$ be the number of general complete intersection subvarieties $R_{i,j} \neq R$ met by $f(C_s')$, and add one to this number if $\dim R' = i$ and $f(C_s')$ meets $R'$.  As $-K_X \cdot \alpha + \dim X - 4 = -K_X \cdot \alpha$, while  $-K_X \cdot \sum_s f_*[C_s'] \leq -K_X \cdot \alpha - 2$ by ampleness of the anticanonical divisor, %
for some $s$ we have $3a_{s,0} + 2a_{s,1} + a_{s,2} \geq -K_X \cdot f_*[C_s'] + 1.$
In this case, however, equality holds by Lemma \ref{lem: mbb precurser}, $C_s'$ is irreducible, and $f|_{C_s'}$ is a general free curve.  Since $f(C_s')$ must meet a subvariety $Y \subset X$ with $a(Y, -K_X) > 1$, by generality of $f|_{C_s'}$ we see that $Y$ would be a divisor, contradicting Lemmas \ref{lem: finitely many planes} and \ref{lem: Picard rank 2 subvarieties higher a-invar}.
\end{proof}

\section{$a$-covers}\label{section:a covers}

We recall the notion of $a$-covers on a smooth Fano varieties:
\begin{defn}\label{a-cover}
    Let $X$ be a smooth Fano variety.
    We say that a dominant generically finite morphism $f\colon Y\rightarrow X$ of degree at least $2$ is an $a$-\textit{cover} if $Y$ is smooth and $a(Y, -f^{*}K_{X}) = a(X, -K_{X}) = 1$.
\end{defn}

For an $a$-cover $f\colon Y \rightarrow X$, we let $\phi_{K_Y - f^*K_X}\colon Y \dashrightarrow B$ be the Iitaka fibration for the adjoint divisor $K_Y - f^*K_X$.  By resolving $Y$, we often assume $\phi_{K_Y - f^*K_X}$ is a morphism.  In this section, we prove Theorem \ref{classification a-covers} describing $a$-covers on smooth Fano varieties of coindex $3$:

\begin{thm}\label{classification a-covers}
    Let $X$ be a smooth Fano variety of coindex $3$ and dimension $n\ge 4$.
    Suppose that $X$ is not a Fano $4$-fold of product type.
    Let $f\colon Y\rightarrow X$ be an $a$-cover, and let $\phi \colon Y\dashrightarrow B$ be the Iitaka fibration for $K_{Y}-f^{*}K_{X}$. 
    Then either 
\begin{itemize}
    \item $n = 4$ and $\phi$ is birational to the base change of a family of $-K_X$-conics; or %
    \item $\rho(X) = 2$ and $\phi$ is birational to the base change of a Mori fibration on $X$.
\end{itemize}\end{thm}

To prove Theorem \ref{classification a-covers}, we relate $a$-covers of Fano varieties of index larger than $1$ to $a$-covers of their fundamental divisors.  We will determine which $a$-covers of fundamental divisors extend to $a$-covers of $X$ using results from Section \ref{section: low degree curves}.

\begin{thm}[\cite{exceptional_sets}, Theorem 3.15]\label{thm:a-cover slice}
Let $X$ be a smooth Fano variety of index $r_X \geq 2$.  Suppose the fundamental linear system $W = |\frac{-1}{r_X}K_X|$ is basepoint free, and let $H \in W$ be general.  If $f : Y \rightarrow X$ is an $a$-cover and $\kappa(K_Y - f^*K_X) < \dim X -1$, the restriction
$$Z = f^{-1}(H), \hspace{.5cm} g = f|_Z: Z \rightarrow H$$
is an $a$-cover of $H$ such that $\kappa(K_Z -g^*K_H) = \kappa(K_Y - f^*K_X)$.
\end{thm}

This allows us to employ a characterization of $a$-covers of Fano $3$-folds with very ample anticanonical divisors \cite{beheshti2020moduli}.  When $-K_X$ is very ample, this implies the Iitaka dimension of $K_Y - f^*K_X$ is at least $2$.  We prove this result for arbitrary $X$ in Lemma \ref{lem: a-covers Iitaka dim 1}.  This will require Lemma \ref{lem: specialize Manin components}, which allows us to extend certain results from Section \ref{section: low degree curves} to arbitrary $X$.

\begin{lem}\label{lem: specialize Manin components}
    Let $\pi \colon \mathcal{X} \rightarrow B$ be a family of smooth Fano varieties of coindex $3$ and dimension $n \ge 4$.  Consider a class $\alpha \in \NE(\pi)$ with $-K_{\mathcal{X}/B} \cdot \alpha \leq 4(n - 2)$.  Suppose a component $M \subset \overline{M}_{0,1}(\mathcal{X},\alpha)$ which generically parameterizes free curves has irreducible fibers over general points in $\mathcal{X}$.  Then the fiber of $M$ over any point in $B$ contains a single dominant component.  
\end{lem}

\begin{proof}
    The fiber of $M$ over any point $p \in \mathcal{X}$ is connected.  It remains to show that for any point $b \in B$ and a general point $p \in X = \mathcal{X}_b$, the fiber $M_p$ of $M$ over $p$ contains only smooth points of $\overline{M}_{0,0}(X,\alpha)$ in codimension $1$.  We reduce to considering the locus of curves passing through $p$ and meeting $m = -K_X \cdot \alpha - 3$ general codimension $2$ complete intersection varieties $R_1, \ldots, R_m \subset X$. %
    Let $H$ be the fundamental divisor on $X$.

    Consider a stable map $f \colon C \rightarrow X$ passing through $p$ and each $R_1, \ldots R_m$.  If for every irreducible component $C_i \subset C$, either $f$ contracts $C_i$ or $[f|_{C_i}] \in \overline{M}_{0,0}(X, f_*[C_i])$ lies in a component parameterizing a dominant family of curves, then by Lemma \ref{lem: mbb precurser} $[f]$ must be a smooth point of $\overline{M}_{0,0}(X,\alpha)$.  Otherwise, we may suppose there exists a subvariety $Y \subset X$ with $a(Y,-K_X) > 1$.  As Lemma \ref{lem: mbb contractible divisor} proves irreducibility of the locus of free curves in $\overline{M}_{0,0}(X,\alpha)$ when $X$ has a contractible divisor, we may suppose $\dim X = 4$ and $(Y, H) \cong (\mathbb{P}^2, \mathcal{O}(1))$.  In this case, since an arbitrary line through $p$ would be disjoint from $Y$, we may suppose $C = C_1 \cup C_2$ is a nodal union of two possibly reducible conics, where $f(C_1)$ meets $p$ and $Y$, while $f(C_2) \subset Y$.  A dimension count shows both $f|_{C_1}$ and $f|_{C_2}$ must be general in moduli under these conditions, which implies both $C_i$ are irreducible and $H^1(f^{*}T_X|_{C_2}) = 0$, since the normal bundle of $Y$ in $X$ is a stable vector bundle with first Chern class $-H$.  In particular, $[f]$ is a smooth point of $M$.
\end{proof}

By Theorem \ref{thm:a-cover slice} and \cite[Theorem~5.5]{beheshti2020moduli}, $\kappa(K_Y - f^*K_X) > 0$ for any a-cover $f : Y \rightarrow X$ of a coindex three Fano variety $X$.  It follows that any Manin component $M \subset \overline{M}_{0,1}(X,\alpha)$ has irreducible fibers over general points in $X$.  We will show the converse when $M$ generically parameterizes free curves and $\alpha \not\in \partial\NE(X)$.  %

\begin{lem}\label{lem: a-covers Iitaka dim 1}
    Let $X$ be a smooth Fano variety of coindex $3$ and dimension $n\ge 4$ and let $H$ be the fundamental divisor on $X$. 
    Suppose $X$ is not a Fano $4$-fold of product type.
    Then there are no $a$-covers $f\colon Y\rightarrow X$ with $\kappa(Y, K_{Y}-f^{*}K_{X}) = 1$.
\end{lem}

\begin{proof}
    Let $f\colon Y\rightarrow X$ be an $a$-cover with $Y$ smooth and $\kappa(K_{Y}-f^{*}K_{X}) = 1$.
    We way assume the Iitaka fibration $\phi\colon Y\rightarrow B$ for $K_Y - f^*K_X$ is a morphism by resolving indeterminacies.  Let $W \subset X$ be a general $3$-fold complete intersection of fundamental divisors in $X$.  Note that $W$ is a smooth Fano $3$-fold of Picard rank $\rho(X)$ and genus $g(X)$.  Moreover, by Theorem \ref{thm:a-cover slice}, the restriction of $f$ to $Z = f^{-1}(W)$ is an $a$-cover of $W$ with $\kappa(K_Z - f^*K_W) = 1$ as well.
    \par 
    First, suppose that $H$ is very ample.
    By \cite[Theorem 5.3]{beheshti2020moduli}, $W$ admits a del Pezzo fibration.
    However, this contradicts \cite[Section~12.7]{FanoV_Shafarevich}.  Hence the claim holds when $H$ is very ample.
    
    Next, we consider the case $H$ is not very ample.  In this case, $\rho(X) = 1$ and $g(X) \in \{2,3\}$.  We may assume $\dim X = 4$ by Theorem \ref{thm:a-cover slice}.  By running a relative minimal model program, we will show a family of embedded free rational $H$-quartics on $X$ is covered by a family of free rational curves on $Y$.  %
    Using Lemmas \ref{lem: quartic very free} and \ref{lem: specialize Manin components}, we see that any family of embedded quartic curves on $X$ generically parameterizes very free curves.  This is a contradiction, since $\kappa(K_Z - f^*K_W) > 0$.  %
    
     Let $\phi_{K_Z - f^*K_W} \colon Z \dashrightarrow B$ be the Iitaka fibration for $K_Z - f^*K_W$.  Assume $\phi_{K_Z - f^*K_W}$ is a morphism by resolving indeterminacies.  By definition, the divisor $K_Z - f^*K_W$ is rigid on the general fiber $F$ of $\phi_{K_Z - f^*K_W}$, while $-f^*K_W$ is semiample.  Since $-K_W$ is ample, a $K_Z -f^*K_W$ minimal model program on $F$ terminates with a smooth weak del Pezzo surface $S$ whose canonical model $\overline{S}$ maps to $W$.  By \cite{keel_rat_connected} \cite{xu_rat_connected}, there is a very free rational curve $C$ contained in the smooth locus of $\overline{S}$.
    
    The preimage $\tilde{C}$ of $C$ in $F$ satisfies $(K_Z - f^*K_W)\cdot \tilde{C}= 0$, so that $- f^*K_W\cdot \tilde{C}= -K_F\cdot \tilde{C} \geq 3$.  Hence, the component $M_Z \subset \overline{M}_{0,0}(Z, [\tilde{C}])$ to which $\tilde{C}$ belongs parameterizes a family of free curves on $Z$ that dominates a component $M_W \subset \overline{M}_{0,0}(W, f_*[\tilde{C}])$.  By \cite[Theorem 6.7]{beheshti2020moduli}, if $- f^*K_W\cdot \tilde{C} = -K_W\cdot f_*[\tilde{C}] > 4$, the boundary of $M_W$ parameterizes reducible curves whose components are free curves.  As $M_Z$ is proper, the induced map $f_* \colon M_Z \rightarrow M_W$ is surjective.  Hence, if $- f^*K_W \cdot \tilde{C}> 4$ we may degenerate $\tilde{C}$ to a union of free curves of lower degree on $Z$, at least one of which is very free on $F$.  Thus, we may assume $-f^*K_W \cdot \tilde{C}\in \{3,4\}$.

    Recall that $W \subset X$ is a smooth fundamental divisor.  Hence, $-K_W = H$ and $H\cdot f_*[\tilde{C}] \in \{3,4\}$.  Let $g \colon \mathbb{P}^1 \rightarrow W$ be a parameterization of $f(\tilde{C})$ and $i \colon W \rightarrow X$ be the inclusion mapping.  Since $\tilde{C} \subset Z$ is a free curve and the normal bundle of $Z = f^{-1}(W) \subset Y$ is ample on $\tilde{C}$, a general deformation $h \colon \mathbb{P}^1 \rightarrow X$ of $i \circ g$ lifts to $Y$.  As this lifting must be contracted by the Iitaka fibration $\phi\colon Y\rightarrow B$, both virtual normal bundles $N_h$ and $N_{i \circ g}$ have a trivial quotient.  However, the short exact sequence
    $$ 0 \rightarrow N_g \rightarrow N_{i \circ g} \rightarrow \mathcal{O}(g^*H) \rightarrow 0$$
    shows $N_{i \circ g}$ has two positive summands.  %
    By \cite[Theorem 3.14]{Kollar1996}, general deformations of $i \circ g$ are embeddings.  Hence, there is a component $M_Y \subset \overline{M}_{0,0}(Y)$ which dominates a component $M_X \subset \overline{M}_{0,0}(X)$ that generically parameterizes embedded cubics or quartics.  Note that the normal bundle of a general curve parameterized by $M_X$ is globally generated but not ample.  In particular, by Lemmas \ref{lem: quartic very free} and \ref{lem: specialize Manin components}, $M_X$ cannot parameterize $H$-quartics.

    Suppose $M_X$ parameterizes $H$-cubics.  A positive dimensional family of curves parameterized by $M_X$ pass through a general point $p \in X$ and meet a general complete intersection curve $C \subset X$ at a fixed point $q \in X$.  By varying $p$, we may assume $q \in X$ is general as well.  By bend-and-break, the boundary of $M_X$ parameterizes a stable curve which meets $p$ and $q$ along distinct irreducible components $C_p, C_q$.  At least one of $C_p, C_q$ is an $H$-line.  Since $M_Y \rightarrow M_X$ is surjective, a family of free $H$-lines also lifts to $Y$.  By gluing these to cubics parameterized by $M_Y$ and smoothing the resulting curve, we conclude a family of embedded $H$-quartics lifts to $Y$, a contradiction. %
\end{proof}

Lemma \ref{lem: a-covers Iitaka dim 1} allows us to describe the general fiber of any $a$-cover $f : Y \rightarrow X$ explicitly in terms of the Iitaka dimension of $K_Y - f^*K_X$.

\begin{lem}\label{a-covers lemma}
Let $X$ be a smooth coindex $3$ Fano variety of dimension $n \geq 4$.  Suppose $X$ is not a Fano $4$-fold of product type.  Let $H$ be the fundamental divisor on $X$.  Suppose $f\colon Y \rightarrow X$ is an $a$-cover, and let $Y_0$ be the closure of a general fiber of $\phi_{K_Y - f^*K_X}$.  One of the following holds:
\begin{itemize}
    \item $\kappa(K_Y - f^*K_X) = 2$ and $(Y_0, f^*H)$  is birationally equivalent to $(Q, \mathcal{O}_{Q}(1))$, where $Q \subset \mathbb{P}^{n-1}$ is an irreducible (and possibly singular) quadric,
    \item $\kappa(K_Y - f^*K_X) = 3$ and $(Y_0, f^*H)$  is birationally equivalent to $(\mathbb{P}^{n-3}, \mathcal{O}_{\mathbb{P}}(1))$.%
\end{itemize}
\end{lem}

\begin{proof}
    By \cite[Theorem 5.5]{beheshti2020moduli}, there are no adjoint rigid $a$-covers of smooth Fano $3$-folds.  Hence, by Theorem \ref{thm:a-cover slice}, $\kappa(K_Y - f^*K_X) > 0$.  By Lemma \ref{lem: a-covers Iitaka dim 1}, $\kappa(K_Y -f^*K_X) \neq 1$ as well.  As $-K_X = (n - 2)H$, our result follows from \cite[Proposition~3.17]{exceptional_sets}.
\end{proof}

We finish the proof of Theorem \ref{classification a-covers} by analyzing normal bundles of low-degree free curves.  

\begin{proof}[Proof of Theorem \ref{classification a-covers}]
    Suppose $f \colon Y \rightarrow X$ is not birational to the base change of a Mori fibration.  If $\kappa(K_Y - f^*K_X) = 3$ and $\dim X = 4$, then $Y$ is birational to a base change of the family of $H$-lines on $X$.  Otherwise, if $\kappa(K_Y - f^*K_X) = 3$ and $\dim X > 4$, Lemma \ref{a-covers lemma} shows a family $M$ of free, embedded conics lift to $Y$.  If $\kappa(K_Y - f^*K_X) = 2$, by Theorem \ref{thm:a-cover slice} we may assume $\dim X = 4$.  In this case a family $M$ of free, embedded cubics lifts to $Y$ instead.

    We may assume $M$ parameterizes curves whose class $\alpha$ does not lie in the boundary of $\NE(X)$.  Otherwise, the Iitaka fibration for $K_Y - f^*K_X$ would be birational to the base change of a Mori fibration because $M$ parameterizes curves that pass through two general points of a general fiber of the Iitaka fibration for $K_Y - f^*K_X$.   
    
    By Theorem \ref{thm: low degree curves}, when $X$ is general there is merely one family of embedded free curves of class $\alpha$ on $X$.  Since this family will contain curves lying in smooth $3$-fold linear sections of $X$, their normal bundles cannot have the required number of trivial summands.  Our result for arbitrary $X$ follows from Lemma \ref{lem: specialize Manin components}.
\end{proof}

\section{Proof of Main Theorems}\label{section: proof of Main thms}

Here we prove Theorem \ref{thm: main result 1}, Theorem \ref{thm: main result 2}, and Corollary \ref{thm: GMC}.  Theorem \ref{classification higher a} is proven in Section \ref{section: subvar with higher a}.  Theorems \ref{MBB} and \ref{thm: component reducible stable maps} are proven in Section \ref{section:mbb}.  

\begin{proof}[Proof of Theorem \ref{thm: main result 1}, Theorem \ref{thm: main result 2}, and Corollary \ref{thm: GMC}]
    Let $X$ be a smooth Fano variety of coindex 3.  We may assume $X$ is not a product variety and $\dim X \geq 4$, as Corollary \ref{cor: product varieties GMC} addresses the remaining cases of Corollary \ref{thm: GMC}.  Let $\alpha \in \NE(X)_{\mathbb{Z}}$.  If $\dim X \geq 5$, we must show $\Mor(\mathbb{P}^1, X, \alpha)$ is irreducible.  This claim is evident from Theorems \ref{classification Picard rank 1} and \ref{classification Picard rank 2} if $\alpha \in \partial\NE(X)$ or if $\alpha \notin \Nef_1(X)$.  Similarly, Theorem \ref{thm: low degree curves}, Theorem \ref{classification a-covers}, and Lemma \ref{lem: specialize Manin components} prove our claim if $H \cdot \alpha = 1$.  We therefore reduce to considering $\alpha \in \Nef_1(X)_{\mathbb{Z}}\setminus \partial\NE(X)$ such that $H \cdot \alpha \geq 2$.

    Theorems \ref{classification a-covers} and \ref{classification higher a} describe all accumulating (non-Manin) components of $\Mor(\mathbb{P}^1, X,\alpha)$; these exist only when $\dim X = 4$ and parameterize either a non-dominant family of curves or multiple covers of $H$-lines.  
    For any Manin component $M \subset \Mor(\mathbb{P}^1, X)$, the universal family $\pi : U \rightarrow M$ has connected fibers under the evaluation map $\text{ev}: U \rightarrow X$.  Hence, if $f : \mathbb{P}^1 \rightarrow X$ is a general map parameterized by $M$, the pullback $f^*\mathcal{T}_X$ will satisfy certain positivity properties guaranteed by \cite[Proposition~3.1]{patel2020moduli}.  In particular, \cite[Theorem 3.14]{Kollar1996} shows $f$ will be an embedding.  %
    It remains to show there is a unique component of $\Mor(\mathbb{P}^1,X,\alpha)$ generically parameterizing embedded free curves.  Equivalently, we must show $\overline{M}_{0,0}(X,\alpha)$ contains only one component generically parameterizing immersed free curves.
    
    Let $M \subset \overline{M}_{0,0}(X, \alpha)$ be an irreducible component which generically parameterizes embedded free curves.  We will show there is exactly one such a component.  When $\rho(X) = 2$ and $g(X) > 11$, \cite{Thomsen1998}, \cite{KimPandharipande2001}, and Lemma \ref{lem: mbb contractible divisor} prove $M$ is unique.  Hence we may assume $X$ does not contain a contractible divisor.  %
    We prove our claim for such $X$ by induction on $d = H \cdot \alpha$. %

    The base cases of our induction are $d = 2, 3$ when $\dim X = 4$ and $d = 1,2$ when $\dim X \geq 5$.  Theorem \ref{thm: low degree curves}, Theorem \ref{classification a-covers}, and Lemma \ref{lem: specialize Manin components} prove $\Mor(\mathbb{P}^1,X,\alpha)$ contains a unique Manin component in these cases.  Therefore, $M$ is unique as well.  
    
    For the inductive step, we study the boundary of $M$.  By Movable bend-and-break (Theorem \ref{MBB}), $M$ contains a stable map $f : C_1 \cup C_2 \rightarrow X$ such that $f|_{C_i}$ is free for each $i$.  Since any family of free curves dominates $X$, one can apply Movable bend-and-break repeatedly until one obtains a chain $f : C_1 \cup \ldots \cup C_d \rightarrow X$ of $H$-lines.  
    \begin{enumerate}[(i)]
            \item Assume $\rho(X) = 1$.
            Smoothing subchains of $C_1 \cup \ldots \cup C_d$ of length $2$ and $d-2$, we obtain the nodal union of two free curves $g\colon C_{1}' \cup C_{2}'\rightarrow X$ such that $H\cdot g_{*}C_{1}' = 2$ and $H\cdot g_{*}C_{2}' = d-2$.  For each $i$, $g|_{C_i'}$ is parameterized by a Manin component of $\overline{M}_{0,0}(X, g_*C_i')$.  By induction on $d$, $[g] \in \overline{M}_{0,0}(X,\alpha)$ lies in the image of a unique component of $\overline{M}_{0,1}(X, g_*C_1') \times_X \overline{M}_{0,1}(X, g_*C_2')$.

            \item Assume $\rho(X) = 2$ and $g(X) \le 11$.
            Smoothing a subchain of $C_1 \cup \ldots \cup C_d$ length $d-1$, we obtain the nodal union of two free curves $g\colon C_{1}' \cup C_{2}'\rightarrow X$ such that $H\cdot g_{*}C_{1}' = 1$, $H\cdot g_{*}C_{2}' = d-1$, and $g_{*}C_{2}' \notin \partial\NE(X)$.
            Since $\overline{M}_{0,0}(X, g_{*}C_{1}')$ is irreducible, one can also show the claim by induction on $d$. \qedhere
        \end{enumerate}
    \end{proof}

\bibliographystyle{alpha}
\bibliography{mybiblio}
\end{document}